\documentclass{amsart}
\usepackage[english]{babel}
\usepackage[utf8]{inputenc}
\usepackage[T1]{fontenc}
\usepackage{amsfonts,amsmath,amssymb,amsthm}
\usepackage{graphicx}
\usepackage{algorithm}
\usepackage[noend]{algpseudocode}
\usepackage{color}
\usepackage{fullpage}

\newcommand{\RR}{\mathbb R}
\newcommand{\NN}{\mathbb N}
\newcommand{\dec}{\mathcal{D}}

\newcommand{\supp}{\operatorname{supp}}
\newcommand{\cat}{\operatorname{{\bf cat}}}
\newcommand{\ucat}{\operatorname{{\bf ucat}}}

\newcommand{\gcat}{\operatorname{{\bf gcat}}}

\newcommand{\pv}{\operatorname{PV}}

\newtheorem{theorem}{Theorem}[section]
\newtheorem{proposition}[theorem]{Proposition}
\newtheorem{observation}[theorem]{Observation}

\newtheorem{lemma}[theorem]{Lemma}
\newtheorem{corollary}[theorem]{Corollary}
\newtheorem*{corollary*}{Corollary}
\newtheorem{step}{Step}

\theoremstyle{definition}
\newtheorem*{definition}{Definition}
\newtheorem*{conjecture}{Conjecture}
\newtheorem*{notation}{Notation}
\newtheorem*{example}{Example}
\newtheorem{question}{Question}
\newtheorem*{convention}{Convention}

\theoremstyle{remark}
\newtheorem*{remark}{Remark}

\author{Dejan Govc}
\address{Institute of Mathematics, Physics and Mechanics, Jadranska 19, 1000 Ljubljana, Slovenia}
\email{dejan.govc@imfm.si}
\title{Unimodal Category and the Monotonicity Conjecture}

\begin{document}



\begin{abstract}
We completely characterize the unimodal category for functions $f:\RR\to[0,\infty)$ using a decomposition theorem obtained by generalizing the sweeping algorithm of Baryshnikov and Ghrist. We also give a characterization of the unimodal category for functions $f:S^1\to[0,\infty)$ and provide an algorithm to compute the unimodal category of such a function in the case of finitely many critical points.

We then turn to the monotonicity conjecture of Baryshnikov and Ghrist. We show that this conjecture is true for functions on $\RR$ and $S^1$ using the above characterizations and that it is false on certain graphs and on the Euclidean plane by providing explicit counterexamples. We also show that it holds for functions on the Euclidean plane whose Morse-Smale graph is a tree using a result of Hickok, Villatoro and Wang.
\end{abstract}

\maketitle

\section{Introduction}

The unimodal category, introduced by Baryshnikov and Ghrist \cite{baryshnikov}, is an invariant of positive functions $f:X\to[0,\infty)$, analogous to the classical concepts of Lusternik-Schnirelmann category and geometric category \cite{cornea}. Whereas the latter two correspond to the minimal number of pieces of a certain type a topological space can be decomposed into, the unimodal category concerns decompositions of functions.

Motivated by statistics, a function $u:X\to[0,\infty)$ is said to be {\em unimodal} if, informally, it has a single local maximum region. The unimodal category $\ucat(f)$ is then defined as the minimum number $n$ of unimodal functions $u_1,\ldots,u_n:X\to[0,\infty)$ such that $f=\sum_{i=1}^n u_i$ pointwise. Considering $\ell^p$-combinations, $p\in(0,\infty)$, instead of sums ($\ell^1$-combinations) yields the related notion of the unimodal $p$-category $\ucat^p(f)$, which is the minimum number $n$ of unimodal functions $u_1,\ldots,u_n:X\to[0,\infty)$ such that $f=\left(\sum_{i=1}^n u_i^p\right)^{\frac1p}$ pointwise. There is also a corresponding notion of $\ucat^{\infty}(f)$ which corresponds to $\ell^{\infty}$-combinations, i.e. pointwise maxima of functions.

Not much is known about the unimodal category, even in the basic case of $X=\RR^m$, which is the most interesting case from the statistical point of view. For the case $m=1$, Baryshnikov and Ghrist provided a simple algorithm \cite{baryshnikov}, which allows us to compute the unimodal category of any function with only finitely many critical points. Their paper treats the case $m=2$ only partially and concludes with a monotonicity conjecture, which they suggest will play a key role in providing precise bounds for the unimodal category in higher dimensions. Our paper is centered around this conjecture, which can be stated as follows.

\begin{conjecture}[\cite{baryshnikov}, Conjecture 18]
Given a continuous function $f:X\to[0,\infty)$ on a topological space and $0<p_1<p_2\leq\infty$, we have
\[
\ucat^{p_1}(f)\leq\ucat^{p_2}(f).
\]
\end{conjecture}

Computation of $\ucat$ in the case $m=2$ is treated in some more detail by Hickok, Villatoro and Wang in \cite{hickok}, which is focused on those Morse distributions on the plane whose Morse-Smale graphs are trees. For these, the unimodal category is almost completely characterized.

In this paper, which is based on the author's doctoral thesis \cite{govc}, we show that the sweeping algorithm of \cite{baryshnikov} can be generalized to functions of bounded variation $f:\RR\to[0,\infty)$ as a variant of the Jordan decomposition theorem. This can be applied to prove the monotonicity conjecture in the case $X=\RR$ and $X=S^1$. We then show that the conjecture is false if $X$ is a certain kind of graph and, more importantly, if $X=\RR^2$.

\section{Preliminaries and Previously Known Results}\label{prelims}

In this section, for the convenience of the reader, we recall some of the basic concepts we use in the proofs.

\subsection{Unimodal Category}

The concept of unimodal category is formalized as follows.

\begin{definition}[\cite{baryshnikov}]
A continuous function $u:X\to[0,\infty)$ is {\em unimodal} if there is an $M>0$ such that the superlevel sets $u^{-1}[c,\infty)$ are contractible for $0<c\leq M$ and empty for $c>M$.

Let $p\in(0,\infty)$. The {\em unimodal $p$-category} $\ucat^p(f)$ of a function $f:X\to[0,\infty)$ is the minimum number $n$ of unimodal functions $u_1,\ldots,u_n:X\to[0,\infty)$ such that pointwise, $f=(\sum_{i=1}^n u_i^p)^{\frac1p}$. Similarly, the {\em unimodal $\infty$-category} $\ucat^\infty(f)$ of a function $f:X\to[0,\infty)$ is the minimum number $n$ of unimodal functions $u_1,\ldots,u_n:X\to[0,\infty)$ such that pointwise, $f=\max_{1\leq i\leq n} u_i$. In place of $\ucat^1(f)$ we usually write $\ucat(f)$.
\end{definition}

\begin{remark}
Baryshnikov and Ghrist also consider $\gcat(f^{-1}(0,\infty))$ as a natural candidate\footnote{Note that they seem to be using the notation $\supp f$ to denote $f^{-1}(0,\infty)$, i.e. the set-theoretic support of $f$.} for the notion of $\ucat^0(f)$ and claim that $\lim_{p\downarrow 0}\ucat^p(f)=\ucat^0(f)$. We note that while this may be a natural candidate, the equation does actually not hold as stated. For instance, if $f:\RR\to[0,\infty)$ is a function with two peaks such that $f^{-1}(0,\infty)$ is connected, e.g. $f(x)=\max\{0,2-||x|-1|\}$, then $\ucat^p(f)=2$ for all $p\in(0,\infty]$ whereas $\gcat(f^{-1}(0,\infty))=1$.
\end{remark}

Arguably, the most important case to understand is $p=1$, see Lemma \ref{powers} below. Note that $\ucat^p(f)$ can only be finite if $f:X\to[0,\infty)$ is continuous, since continuity is assumed in the definition of unimodality. This motivates the following convention.

\begin{convention}
Unless explicitly noted otherwise, all functions $f:X\to[0,\infty)$ will hereafter be assumed to be continuous.
\end{convention}

Baryshnikov and Ghrist further assume that the functions they study are compactly supported. We do {\bf not} make this assumption. Because of this, our notion of unimodal category slightly differs from theirs. Note that the two notions of $\ucat^p(f)$ agree whenever $f$ is compactly supported. However, if $f$ is not compactly supported, their notion is undefined, while ours may still give a finite answer. Hence, the notion we use is a slight generalization of theirs. (Note that all the main examples and counterexamples considered in the paper will be compactly supported, so they are still valid under the original definition.)

While the definition of $\ucat^p$ makes sense for general topological spaces, it seems to be the most interesting for spaces that are not too pathological. As there are plenty of open questions already in the case when $X$ is a manifold or a CW complex, we will restrict our attention to those.

\subsection{Total Variation and Jordan Decomposition}

The concept of total variation (see e.g. \cite[Chapter 6]{apostol}) for functions $f:\RR\to[0,\infty)$ will be useful, so we recall the basic definitions and results here.

\begin{definition}
The {\em total variation} of $f:\RR\to\RR$ on the interval $[a,b]$ is defined by the formula
\[
V(f;[a,b])=\sup\sum_{i=1}^n |f(x_i)-f(x_{i-1})|,
\]
where the supremum is taken over all partitions $a=x_0<x_1<\ldots<x_n=b$ of the interval $[a,b]$. Similarly, we define the {\em positive variation} of $f$ on the interval $[a,b]$ as
\[
V^+(f;[a,b])=\sup\sum_{i=1}^n \max\{0,f(x_i)-f(x_{i-1})\},
\]
and the {\em negative variation} of $f$ on the interval $[a,b]$ as
\[
V^-(f;[a,b])=\sup\sum_{i=1}^n \max\{0,f(x_{i-1})-f(x_i)\}.
\]
\end{definition}

We will use the following basic facts:

\begin{theorem}\label{variation_props}
Let $V^*$ stand for $V,V^+$ or $V^-$ and let $f:[a,b]\to\RR$. Then the following hold:
\begin{itemize}
\item If $f$ is increasing on $[a,b]$, then $V(f;[a,b])=V^+(f;[a,b])=f(b)-f(a)$ and $V^-(f;[a,b])=0$.
\item If $f$ is decreasing on $[a,b]$, then $V(f;[a,b])=V^-(f;[a,b])=f(a)-f(b)$ and $V^+(f;[a,b])=0$.
\item If $f=\sum_{i=1}^n f_i$, then $V^*(f;[a,b])\leq\sum_{i=1}^n V^*(f_i;[a,b])$.
\item If $a=x_0<x_1<\ldots<x_n=b$, then $V^*(f;[a,b])=\sum_{i=1}^n V^*(f;[x_{i-1},x_i])$.
\item $V^+(f;[a,b])+V^-(f;[a,b])=V(f;[a,b])$ and $V^+(f;[a,b])-V^-(f;[a,b])=f(b)-f(a)$.
\end{itemize}
\end{theorem}

\begin{proof}
The first two facts are obvious. For the third and the fourth fact in the case of $V$, see Theorems 6.9 and 6.11 in \cite{apostol}. For $V^+$ and $V^-$, the idea is completely analogous. The fifth fact is a standard exercise, see \cite{apostol}, Exercise 6.7.
\end{proof}

By undergraduate measure theory, every right continuous increasing function on $\RR$ determines a Borel measure \cite[Section IV.8]{doob}, so $V,V^+$ and $V^-$ can be extended to define measures on $\RR$. We are therefore also allowed to compute these variations over open intervals or even Borel sets, but note that continuity of $f$ implies that $V^*(f;[a,b])=V^*(f;(a,b))$.

These concepts also make sense for $f:J\to\RR$, where $J\subseteq\RR$ is an interval. Functions of bounded variation, i.e. such that $V(f;J)<\infty$, are of particular interest to us. Note that the first two bullet points of Theorem \ref{variation_props} imply that monotone functions over finite intervals have bounded variation. More generally, bounded monotone functions also have bounded variation. Using the fourth bullet point, this also implies that unimodal functions $u:J\to\RR$ have bounded variation, since the domain of any such function can be split into two intervals where the function is monotone and bounded.

Any function $f:J\to\RR$ of bounded variation can be split as the difference of two monotone functions, i.e. it satisfies the following Jordan decomposition theorem:

\begin{theorem}
Suppose $f:J\to\RR$ is of bounded variation. Then $f$ can be expressed as the difference $f=g-h$ of two increasing functions $g,h:J\to\RR$.
\end{theorem}

\begin{proof}
See \cite{apostol}, Theorem 6.15. The result is stated there for closed intervals, but it actually holds in general. The idea is to take $g(x)=V(f;J\cap(-\infty,x))$ and $h(x)=g(x)-f(x)$.
\end{proof}

In Section \ref{real_line}, we use a slightly different decomposition, namely $g(x)=V^+(f;J\cap(-\infty,x))$ and $h(x)=V^-(f;J\cap(-\infty,x))$, assuming $\lim_{x\to-\infty}f(x)=0$. The limit at $-\infty$ can always be subtracted from $f$, so this is not really a restriction. To see that this is indeed a decomposition, use the fifth property of Theorem \ref{variation_props}.

\section{Computing the Unimodal Category}\label{computing}

Given a function $f:X\to[0,\infty)$, the most basic question we are interested in is how to compute its unimodal category $\ucat(f)$. This has been answered for functions $f:\RR\to[0,\infty)$ with finitely many critical points by Baryshnikov and Ghrist \cite[Theorem 11]{baryshnikov}, using a simple sweeping algorithm. We show that this sweeping algorithm can be seen as arising from the Jordan decomposition theorem for functions of bounded variation. This gives a complete answer to the question in the case of $X=\RR$. We can use the same idea to answer the question when $X=S^1$.

\subsection{Real Line and Intervals}\label{real_line}

We show that the method to obtain a minimal unimodal decomposition for a compactly supported continuous function $f:\RR\to[0,\infty)$ described in \cite{baryshnikov} in the case of finitely many critical points actually works for an arbitrary $f:\RR\to[0,\infty)$. First, observe that the problem is only interesting for functions of bounded variation:

\begin{proposition}
Suppose $f:\RR\to[0,\infty)$ has a unimodal decomposition $f=\sum_{i=1}^n u_i$. Then $f$ is of bounded variation on each interval $[a,b]\subseteq\RR$.
\end{proposition}

\begin{proof}
Let $[a,b]$ be an arbitrary interval. Then
\[
V(f;[a,b])=V\left(\sum_{i=1}^n u_i;[a,b]\right)\leq\sum_{i=1}^n V(u_i;[a,b])<\infty,
\]
since unimodal functions are of bounded variation.
\end{proof}

We may therefore restrict our attention to functions of bounded variation. For a complete characterization of functions $f:\RR\to[0,\infty)$ with finite unimodal category, see Theorem \ref{bv_has_ucat} below.

\begin{remark}
The fact that unimodal functions have bounded variation is specific to $\RR$. It is not difficult to construct a unimodal function $u:\RR^2\to[0,\infty)$ which does not have bounded variation. Take any differentiable function $f:(0,1)\to\RR$ whose total variation is unbounded and $0\leq f(x)\leq 1$ for all $x$. Further assume that the limits $\lim_{x\to0}f(x)$ and $\lim_{x\to1}f(x)$ exist. For instance, one might take
\[
f(x)=x(1+\sin\tfrac{\pi}x).
\]
Define $u:[0,1]\times[0,1]\to[0,\infty)$ on $(0,1)\times(0,1)$ by
\[
u(x,t)=tf(x)+(1-t)
\]
and by its unique continuous extension elsewhere. This is a unimodal function but does not have bounded variation. Recall that the total variation of a differentiable function can be computed as the integral of the norm of its gradient:
\[
V(u;[0,1]\times[0,1])=\int_0^1\int_0^1\sqrt{u_x^2+u_t^2}\mathrm dx\mathrm dt\geq\int_0^1\int_0^1|u_x|\mathrm dx\mathrm dt=\int_0^1t\mathrm dt\int_0^1|f'(x)|\mathrm dx=\tfrac12V(f;(0,1))=\infty.
\]
If desired, $u$ can be extended to a unimodal function on $\RR^2$.
\end{remark}

To see that the method of Ghrist and Baryshnikov generalizes, we first suitably modify Proposition 10 of \cite{baryshnikov}. First, recall Definition 7 of \cite{baryshnikov}:
\begin{definition}
Let $\dec=(u_i)_{i=1}^n$ be a unimodal decomposition of $f:\RR\to[0,\infty)$. An open interval $(x,y)$ is {\em $\dec$-max-free} if it contains no local maxima of any of $u_i$.
\end{definition}

In \cite{baryshnikov}, the authors only use this concept for intervals bounded by local minima of $f$. However, it is more generally applicable:

\begin{proposition}\label{maxfree}
If an {\em arbitrary} open interval $(x,y)$ is $\dec$-max-free, where $\dec=(u_i)_{i=1}^n$ is a unimodal decomposition of $f:\RR\to[0,\infty)$, then
\[
V^-(f;(x,y))\leq f(x).
\]
\end{proposition}

\begin{proof}
The idea is the same as in \cite{baryshnikov}:
\[
V^-(f;(x,y))\leq \sum_{i=1}^n V^-(u_i;(x,y))=\sum_{i=1}^n\max\{0,u_i(x)-u_i(y)\}\leq\sum_{i=1}^n u_i(x)=f(x),
\]
where the first equality uses the max-free condition.
\end{proof}

So, the concept of {\bf forced-max} interval from \cite{baryshnikov} makes sense even if the interval is not bounded by local minima:

\begin{definition}
An interval $(x,y)$ is called {\bf forced-max} (with respect to $f$) if
\[
V^-(f;(x,y))>f(x).
\]
In addition to this, we will call an interval $(x,y)$ {\bf almost forced-max} if $(x,y+\delta)$ is forced-max for each $\delta>0$. If $f:\RR\to[0,\infty)$ is compactly supported, we define $M(f)$ to be the maximum number of pairwise disjoint open intervals that are forced-max with respect to $f$, and $\widetilde M(f)$ to be the maximum number of pairwise disjoint open intervals that are almost forced-max with respect to $f$. (It is important to use open intervals here, as closed intervals with the same endpoints could intersect.)

More precisely, we should separate two cases here: if there is a finite bound to the number of such intervals, the maximum indeed exists and can be realized by a collection of such intervals; however, if there is no upper bound, this does not by itself imply the existence of an infinite collection of such intervals. We resolve this issue in Proposition \ref{realization}, where we show that such a collection does indeed always exist.
\end{definition}

The observation from \cite{baryshnikov} that (almost) forced-max intervals form an ideal in the sense that an interval containing an (almost) forced-max interval is itself (almost) forced-max, remains valid in this context. We now show that the numbers $M(f)$ and $\widetilde M(f)$ always agree. For this reason we only speak of $M(f)$ in the rest of the text.

\begin{proposition}\label{almost}
If $f:\RR\to[0,\infty)$ is a compactly supported function of bounded variation, we have
\[
\widetilde M(f)=M(f).
\]
\end{proposition}

\begin{proof}
The inequality $\widetilde M(f)\geq M(f)$ trivially follows from the definitions. It remains to show that $\widetilde M(f)\leq M(f)$. We do this by an inductive construction. Suppose $(x_0,y_0),...,(x_{n-1},y_{n-1})$ is an arbitrary family of pairwise disjoint almost forced-max intervals. We may assume without loss of generality that $x_0<x_1<\ldots<x_n:=y_{n-1}$ and that $y_i=x_{i+1}$ for $i=1,\ldots,n-1$ and $x_0=\inf\supp f$, $x_n=\sup\supp f$. We claim that it is possible to choose $\delta_1,\ldots,\delta_{n-1}>0$ such that the intervals $(x_0,x_1+\delta_1),(x_1+\delta_1,x_2+\delta_2),\ldots,(x_{n-1}+\delta_{n-1},x_n)$ are forced-max and clearly these are still disjoint.

We prove this by induction from $n$ downwards. For the base case, we do not actually change anything, we just note that $(x_{n-1},x_n)$ is already forced-max. This is because $V^-(f;[x_{n-1},x_n])=V^-(f;[x_{n-1},x_n+\delta])$, since $f$ is zero on $[x_n,\infty)$. Assuming that $\delta_k,\ldots_,\delta_{n-1}$ (possibly none of them in the case $k=n$) have already been chosen such that $(x_{k-1},x_k+\delta_k),\ldots,(x_{n-1}+\delta_{n-1},x_n)$ are forced-max, we observe that since $V^-(f;[x_{k-1},x_k+\delta_k])>f(x_{k-1})$ and $f$ is continuous, a $\delta_{k-1}>0$ exists such that $V^-(f;[x_{k-1}+\delta_{k-1},x_k+\delta_k])>f(x_{k-1}+\delta_{k-1})$. The interval $(x_{k-2},x_{k-1}+\delta_{k-1})$ thus becomes genuinely forced-max.
\end{proof}

Furthermore, Theorem 11 in \cite{baryshnikov} states that $\ucat(f)=M(f)$ for functions that are nice enough. This can also be generalized to our context. In fact, we are going to explicitly construct a unimodal decomposition of $f$. This is achieved by a recursive construction, which generalizes the sweeping algorithm from \cite{baryshnikov}. For convenience, we first construct a minimal unimodal decomposition in the compactly supported case. The general case is then treated in Proposition \ref{intervals} below.

\begin{theorem}\label{sweeping}
If $f:\RR\to[0,\infty)$ is compactly supported and of bounded variation, then
\[
\ucat(f)=M(f)
\]
holds. If $M(f)=n<\infty$, then an explicit minimal unimodal decomposition $(u_i)_{i=1}^n$ is obtained by the following procedure. First, extend $f$ to $\overline{\RR}=[-\infty,\infty]$ and define $g,h:\overline{\RR}\to[0,\infty)$ by
\[
g(x)=V^+(f;(-\infty,x])\qquad\text{and}\qquad h(x)=V^-(f;(-\infty,x]).
\]
Recursively define a finite sequence\footnote{We use the standard convention that $\inf\emptyset=\infty$ and $(a,b)=\emptyset$ if $a\geq b$.} $(x_i)_{i=0}^n$:
\begin{align*}
x_0&=-\infty,\\
x_i&=\inf\{x\mid V^-(f;(x_{i-1},x))>f(x_{i-1})\},\qquad i=1,\ldots,n,\\
x_{n+1}&=\infty.
\end{align*}
Finally, define $u_i:\RR\to[0,\infty)$ by
\[
u_i(x)=\begin{cases}
0;&x\leq x_{i-1},\\
g(x)-g(x_{i-1});&x\in[x_{i-1},x_i],\\
h(x_{i+1})-h(x);&x\in[x_i,x_{i+1}],\\
0;&x\geq x_{i+1}.
\end{cases}
\]
\end{theorem}

\begin{proof}
The inequality $\ucat(f)\geq M(f)$ follows directly from the definitions. This is because given any unimodal decomposition of $f$, each forced-max interval must contain a local maximum of some unimodal summand in this decomposition, hence there are at least as many summands as there are disjoint forced-max intervals. In particular, if $M(f)=\infty$, we immediately have $\ucat(f)=\infty$.

Now, assume $M(f)<\infty$ and let $M(f)=n\in\NN_0$. Note that $f=g-h$ and that $g$ and $h$ are increasing functions. This latter fact implies that each function $u_i$ is increasing on $(-\infty,x_i]$ and decreasing on $[x_i,\infty)$ and therefore unimodal.

To see that the functions $u_i$ are well defined, we have to show that $g(x_i)-g(x_{i-1})=h(x_{i+1})-h(x_i)$ holds for $1\leq i\leq n$. In fact, we can show more, i.e. that $g(x_{i-1})=h(x_i)$. First, observe that $V^-(f;[x_{i-1},x_i])=f(x_{i-1})$. If $x_i<\infty$, this is true by definition of $x_i$, since $(x_{i-1},x_i+\delta)$ is forced-max if $\delta>0$ and is not forced-max if $\delta<0$. If $x_i=\infty$, this is trivial in the case $x_{i-1}=\infty$, otherwise it follows by the definition of $x_i$ that $(x_{i-1},x_i)$ is not forced-max, so $V^-(f;[x_{i-1},x_i])\leq f(x_{i-1})$, and $V^-(f;[x_{i-1},x_i])\geq f(x_{i-1}) - f(x_i)= f(x_{i-1})$ is true by definition of $V^-$. We therefore have:
\begin{align*}
g(x_{i-1})&=V^+(f;(-\infty,x_{i-1}])\\
&=f(x_{i-1})+V^-(f;(-\infty,x_{i-1}])\\
&=V^-(f;[x_{i-1},x_i])+V^-(f;(-\infty,x_{i-1}])\\
&=V^-(f;(-\infty,x_i])\\
&=h(x_i).
\end{align*}
This establishes that the functions $u_i$ are well defined. Clearly, they are also unimodal. Finally, we also have $f=\sum_{i=1}^n u_i$. To see this, observe that for $x\in[x_{i-1},x_i]$, we have:
\[
\sum_{i=1}^n u_i(x)=g(x)-g(x_{i-1})+h(x_i)-h(x)=f(x).
\]
This concludes the proof.
\end{proof}

A corollary of this is that in the case $M(f)=\infty$, we can actually find an infinite set of pairwise disjoint almost forced-max intervals.

\begin{corollary}\label{realization}
Suppose $f:\RR\to[0,\infty)$ is compactly supported and of bounded variation. If $M(f)=\infty$, there is an infinite set of disjoint almost forced-max intervals with respect to $f$.
\end{corollary}

\begin{proof}
The fact that $M(f)=\infty$ implies that the recursively defined sequence
\begin{align*}
x_0&=-\infty,\\
x_i&=\inf\{x\mid V^-(f;(x_{i-1},x))>f(x_{i-1})\},\qquad i\in\NN
\end{align*}
is strictly increasing (otherwise, using the theorem we would have a finite unimodal decomposition). This gives us infinitely many almost-forced max intervals.
\end{proof}

From the explicit construction above, it also follows that a function with finite unimodal category must be well-behaved near the boundary of its support:

\begin{corollary}\label{increasing}
Suppose $f:\RR\to[0,\infty)$ has finite unimodal category and $\supp f=[a,b]\subseteq\RR$. Then there is an $\epsilon>0$ such that $f$ is increasing on $[a,a+\epsilon]$ and decreasing on $[b-\epsilon,b]$.
\end{corollary}

This allows us to completely characterize the functions with finite unimodal category.

\begin{theorem}\label{bv_has_ucat}
A compactly supported function $f:\RR\to[0,\infty)$ of bounded variation has finite unimodal category if and only if:
\begin{itemize}
\item $f^{-1}(0,\infty)=\bigcup_{j=1}^m(a_j,b_j)$ for some $-\infty\leq a_1<b_1\leq a_2<b_2\leq\ldots\leq a_m<b_n\leq\infty$, and
\item there is an $\epsilon>0$ such that $f$ is increasing on $[a_j,a_j+\epsilon]$ and decreasing on $[b_j-\epsilon,b_j]$ for each $j=1,\ldots,m$.
\end{itemize}
\end{theorem}

\begin{proof}
The forward implication is obvious from Theorem \ref{sweeping} and Corollary \ref{realization}. Conversely, suppose the two bullet points are satisfied. We can decompose $f$ uniquely as $f=\sum_{j=1}^m f_j$, where $f_j$ is supported in the interval $[a_j,b_j]$. It suffices to show that each of these functions $f_j$ has a finite unimodal decomposition, so without loss of generality, we can assume that $m=1$ and $f_1=f$. Let $a=a_1$ and $b=b_1$. We can assume that $\epsilon$ is small enough so that the interval $[a+\epsilon,b-\epsilon]$ is nonempty. This interval is also compact, so $f$ achieves a minimum on it, say $f(x)\geq\eta>0$ for all $x\in[a+\epsilon,b-\epsilon]$. Since $f$ is of bounded variation, there is a $q\in\NN$ such that $V^-(f;[a+\epsilon,b-\epsilon])<q\eta$. This allows us to show that $f$ has finite unimodal category. Suppose $(x_1,y_1),\ldots,(x_k,y_k)$ is a set of disjoint intervals. At most one of these intervals can contain the point $a+\epsilon$ and at most one of them can contain the point $b+\epsilon$. Therefore, all the other intervals are contained in one of the intervals $(-\infty,a+\epsilon)$, $(a+\epsilon,b-\epsilon)$, $(b-\epsilon,\infty)$. Now observe that any interval $(x,y)\subseteq(-\infty,a+\epsilon)$ or $(x,y)\subseteq(b-\epsilon,\infty)$ cannot be forced-max (since $f$ is increasing on the first interval and decreasing on the second one). On the other hand, among the aforementioned intervals, there are less than $q$ disjoint forced-max intervals $(x'_l,y'_l)\subseteq(a+\epsilon,b-\epsilon)$, otherwise we would have:
\[
V^-(f;(x,y))\geq\sum_{l=1}^q V^-(f;(x'_l,y'_l))>\sum_{l=1}^q f(x'_l)\geq q\eta,
\]
which is not the case. Therefore $f$ has at most $q+2$ forced-max intervals and the proof is complete.
\end{proof}

\begin{example}
To illustrate the theorems on a function with an infinite set of local maxima, let $C\subseteq[0,1]$ be the usual middle-thirds Cantor set and define
\[
f(x)=\max\{0,\tfrac12-d(x,C)\}.
\]
\begin{center}
\includegraphics[width=340pt]{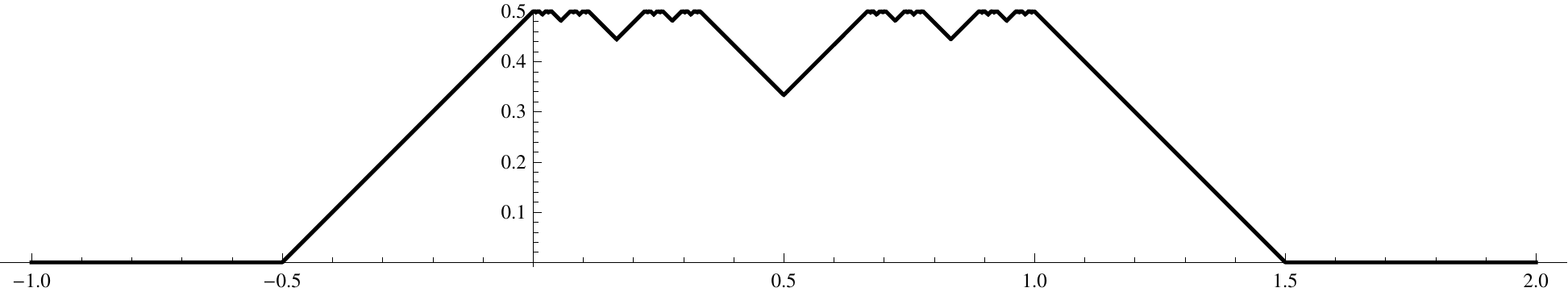}
\end{center}
Observe that $f^{-1}(0,\infty)=(-\frac12,\frac32)$ and that $f$ is increasing on $[-\frac12,0]$ and decreasing on $[1,\frac32]$. This already implies that $f$ has finite unimodal category. In fact, using the explicit construction of Theorem \ref{sweeping}, we can show that $\ucat(f)=2$. To see this, observe that $(-\infty,x)$ is forced-max if and only if $x>0$, so in the construction, we must take $x_1=0$. Next, observe that 
\[
V^-(f;(0,1))=\frac16+\frac2{3\times 6}+\frac4{3^2\times 6}+\ldots=\frac12.
\]
Since $f$ is strictly decreasing on $[1,\frac32]$, for any $x>1$, we have $V^-(f;(0,x))>\frac12$. Therefore, we must take $x_2=1$. Finally, $(1,\infty)$ is not forced-max. Hence, the decomposition obtained by sweeping has precisely two unimodal summands $u_1$ and $u_2$. These can be graphed as follows:
\begin{center}
\includegraphics[width=340pt]{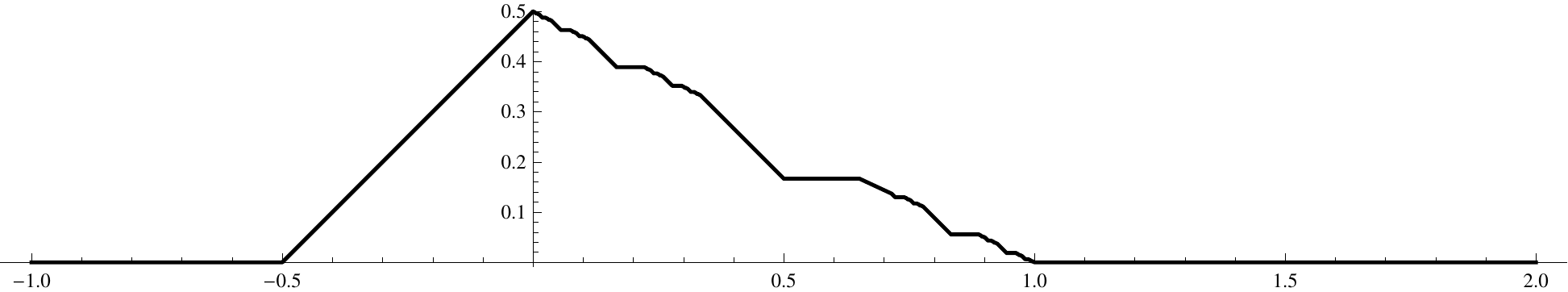}
\end{center}
\begin{center}
\includegraphics[width=340pt]{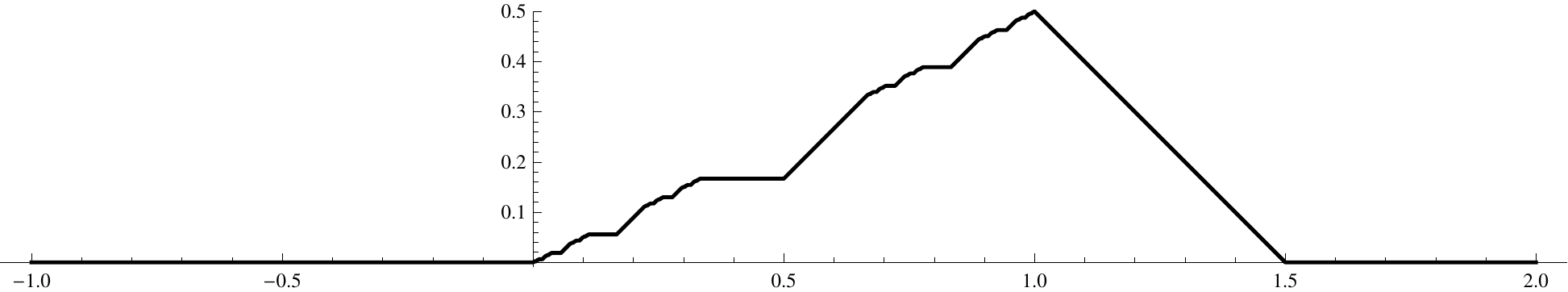}
\end{center}
The construction used in the theorem has a nice geometric interpretation:
\begin{itemize}
\item Plot the positive variation $g$ and the negative variation $h$.
\item Draw a broken line, as follows. For each $i\in\{1,2,\ldots,n\}$, draw a horizontal segment (or infinite ray) from $(x_{i-1},g(x_{i-1}))$ to $(x_i,h(x_i))$  and a vertical segment from $(x_i,h(x_i))$ to $(x_i,g(x_i))$. Conclude with a horizontal ray from $(x_n,h(x_n))$ to $(x_{n+1},g(x_{n+1}))$.
\item This divides the area between $g$ and $h$ into $2n$ pieces. For each $i=1,\ldots,n$, there are two pieces between the $(i-1)$-th and the $i$-th horizontal line (where the line at the bottom is indexed by $0$), and they are divided by a vertical line. Flipping the second piece and translating both pieces downward so that their bottom edge is on the $x$-axis yields a set bounded by two curves: the $x$-axis and the graph of $u_i$.
\end{itemize}
Note that in the second step, the horizontal segment drawn for each $i$ is the longest horizontal segment beginning at the endpoint of the $(i-1)$-th vertical segment that does not cross the graph of $h$. Similarly, each vertical segment is the longest vertical segment beginning at the given point that does not cross the graph of $g$. In the given example, the first two steps are pictured as follows:
\begin{center}
\includegraphics[width=340pt]{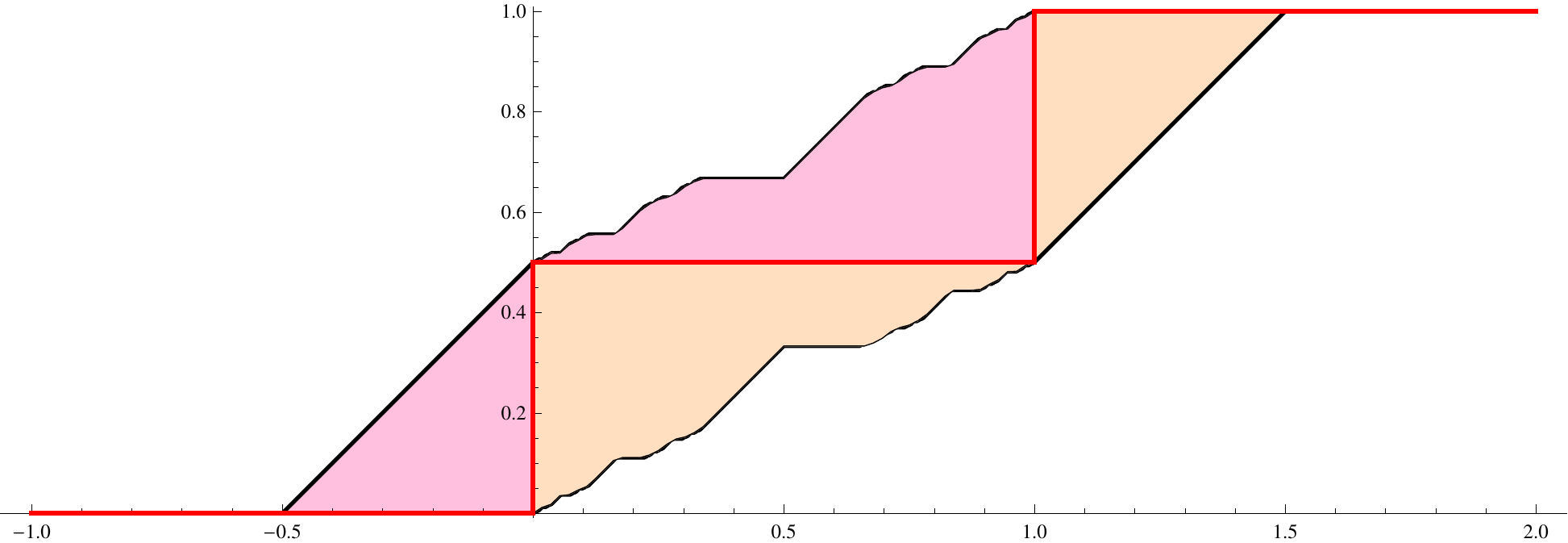}
\end{center}
The third step consists of assembling the first two pieces to obtain $u_1$ and the second two pieces to obtain $u_2$, whose graphs are pictured above. (The orange pieces are flipped in the process.)
\end{example}

Now we turn to functions $f:J\to[0,\infty)$, where $J\subseteq\RR$ is an interval. In this case, we define $M(f)$ as the maximum number of forced-max intervals contained in $J$. Note that again, this agrees with the maximum number of almost forced-max intervals $\widetilde M(f)$. The proof is the same as for Proposition \ref{almost}, but note that the notion of ``almost forced-max'' depends on the ambient, as the enlarged intervals must be contained in $J$. The construction of Theorem \ref{sweeping} can be used to construct a minimal unimodal decomposition\footnote{This actually allows us to compute $\ucat(f)$ of $f:X\to[0,\infty)$ for any set $X\subseteq\RR$: if $f^{-1}(0,\infty)$ has infinitely many components, $\ucat(f)$ is infinite, otherwise $f^{-1}(0,\infty)$ is a finite union of intervals, each of which can be treated separately.} of $f$. Using an appropriate homeomorphism $h$, first map this interval to an interval with endpoints $0$ and $1$ (note that this does not affect unimodality and hence does not change $\ucat$). Now, since $f\circ h^{-1}$ is of bounded variation, it has limits at $0$ and $1$, so we may extend it uniquely to a function $\bar f:[0,1]\to\RR$. Finally, we can extend this function to $\hat f:\RR\to[0,\infty)$ by
\begin{equation}\label{extension}
\hat f(x)=\begin{cases}
(x+1)\bar f(0);& x\in[-1,0],\\
\bar f(x);&x\in[0,1],\\
(2-x)\bar f(1);&x\in[1,2],\\
0;&x\notin[-1,2].\\
\end{cases}
\end{equation}
We claim that using the construction of Theorem \ref{sweeping} on $\hat f$ gives us a minimal unimodal decomposition of $f$:

\begin{proposition}\label{intervals}
Let $J\subseteq\RR$ be an interval and $f:J\to[0,\infty)$ a function of bounded variation. Then
\[
\ucat(f)=\ucat(\hat f),
\]
where $\hat f$ is as defined in Equation \eqref{extension}. A minimal unimodal decomposition of $f$ is obtained by using the construction of Theorem \ref{sweeping} on $\hat f$, restricting the obtained unimodal summands to $h(J)$ and composing with $h$.

Furthermore, we either have $M(\hat f|_{(-\infty,1]})=M(\hat f)-1$ or $M(\hat f|_{(-\infty,1]})=M(\hat f)$. The last summand (i.e. the one whose maximum appears at the largest $x$) in this minimal unimodal decomposition of $f$ is increasing if and only if $M(\hat f|_{(-\infty,1]})=M(\hat f)-1$.
\end{proposition}

\begin{proof}
Restriction of a unimodal function to an interval is still unimodal, therefore $\ucat(f)\leq\ucat(\hat f)$. Conversely, suppose $u_1,\ldots, u_n$ is a unimodal decomposition of $f$. Then construct $\hat u_i$ from $u_i$ using the same procedure used to construct $\hat f$ from $f$. This yields a unimodal decomposition of $\hat f$, so $\ucat(f)\geq\ucat(\hat f)$. Note that the output of the ``sweeping algorithm'' for $f$ on $[0,1]$ does not depend on left-most interval where the input function $\hat f$ increases, so we can sweep directly over $[0,1]$. (The reader is warned that $M(f)=\ucat(f)$ may no longer be true in this case.)

Next we establish the bounds for $M(\hat f|_{(-\infty,1]})$. It is clear that $M(\hat f|_{(-\infty,1]})\leq M(\hat f)$.

In the other direction, let $(x_1,y_1),\ldots,(x_n,y_n)$ be a sequence of forced-max intervals that realizes $M(\hat f)$, ordered by the size of the endpoints. Observe that $x_n<1$, otherwise $(1,2)$ would be a forced-max interval, which cannot be the case, since $\hat f$ is decreasing there. Therefore $(x_1,y_1),\ldots,(x_{n-1},y_{n-1})$ are forced-max intervals for $\hat f|_{(-\infty,1]}$. This proves that $M(\hat f|_{(-\infty,1]})\geq M(\hat f)-1$.

Now, let $u_1,u_2,\ldots,u_n$ be the decomposition of $\hat f$ obtained by sweeping, where the summands are ordered as in the construction, and suppose $u_n|_{(-\infty,1]}$ is increasing. Then $\ucat(\hat f)=M(\hat f)=n$. We claim that $M(\hat f|_{(-\infty,1]})=n-1$. Suppose that this does not hold. Then there are forced-max intervals $(x_1,y_1),\ldots,(x_n,y_n)$ for $\hat f|_{(-\infty,1]}$. By Proposition \ref{maxfree} this means that for each unimodal decomposition of $\hat f|_{(-\infty,1]}$ and for each $i$, there is a summand that achieves its maximum in $(x_i,y_i)$. This contradicts the fact that $u_n|_{(-\infty,1]}$ is increasing.

Conversely, suppose that $M(\hat f|_{(-\infty,1]}) = n-1$. Define $x_0,x_1,\ldots,x_{n+1}$ as in Theorem \ref{sweeping}. The intervals $(x_0,x_1),\ldots,(x_{n-1},x_n)$ are then almost forced-max with respect to $\hat f$. Since there are $n$ of them, at least one is not almost-forced max with respect to $\hat f|_{(-\infty,1]}$. In fact, this is true for any interval whose closure intersects $[1,\infty)$. Therefore, $(x_{n-1},x_n)$ is not almost-forced max with respect to $\hat f|_{(-\infty,1]}$. Furthermore, there is at most one such interval, since $\ucat(f)=n$. We can conclude that $x_{n-1}<1\leq x_n$. Now, simply recall that $u_n$ is increasing on $(x_{n-1},x_n)$ by construction, and the proof is complete.
\end{proof}

The following property of unimodal decompositions is sometimes useful.

\begin{proposition}\label{modsweep}
Suppose $f:[0,1]\to[0,\infty)$, $2\leq\ucat(f)\leq\infty$ and $u_1,\ldots,u_n$ is the unimodal decomposition obtained by sweeping. Then $u_1(1)=\ldots=u_{n-2}(1)=0$ and there exist unimodal functions $u_{n-1}'$ and $u_n'$ such that $u_1,\ldots,u_{n-2},u_{n-1}',u_n'$ is a unimodal decomposition of $f$ and $u_{n-1}'(1)=0$. Let
\[
a=\max\{x\mid u_{n-1}|_{[0,x]}\text{ is increasing}\}.
\]
It is possible to choose $u_{n-1}'$ so that it only differs from $u_{n-1}$ in the interval $(a,1]$ and if $u_n$ was increasing, $u_n'$ is again increasing.\footnote{Note that $u_n$ is automatically increasing if $u_{n-1}(1)\neq0$.}
\end{proposition}

\begin{proof}
The fact that $u_1(1)=\ldots=u_{n-2}(1)=0$ follows directly from the construction in the proof of Theorem \ref{sweeping}. Let $c:=u_{n-1}(1)$ and let $m=u_{n-1}(a)$. If $c=0$, there is nothing to do, so we shall assume that $c>0$. The sweeping construction then implies that $u_n$ is increasing. It also implies that $c<m$, so for $x\in[a,1]$ we may define
\[
u_{n-1}'(x)=m\frac{u_{n-1}(x)-c}{m-c}
\]
and
\[
u_n'(x)=u_n(x)+c\frac{m-u_{n-1}(x)}{m-c}.
\]
Note that $u_{n-1}'(1)=0$, $u_{n-1}'(a)=m$ and since $u_{n-1}\leq m$, we also have $u_{n-1}'\leq u_{n-1}$ and $u_{n-1}'$ is decreasing in $[a,1]$. Therefore, $u_{n-1}'$ is unimodal. The function $u_n'$ is defined as the sum of two increasing functions, so it is also unimodal. This completes the proof.
\end{proof}

\subsection{Circle}

As in the previous section, we are only interested in functions $f:S^1\to[0,\infty)$ that are continuous and of bounded variation, in the sense that the function $[0,1]\to[0,\infty)$ defined by $t\mapsto f(\exp(2\pi i t))$ has bounded variation. Note that this is again a necessary condition for finiteness of $\ucat$. If $f$ has a zero, the problem immediately reduces to the case of $X=\RR$, so we only deal with functions without zeros. Note that this already implies $\ucat(f)\geq\gcat(S^1)=2$. We now introduce some notations.

\begin{notation}
Suppose $a\in S^1$. In the rest of this section, $\phi_a^+:[0,1]\to S^1$ is defined by $\phi_a^+(t)=a\exp(2\pi it)$ and $\phi_a^-:[0,1]\to S^1$ by $\phi_a^-(t)=a\exp(-2\pi it)$. We also define $f_a^+=f\circ\phi_a^+$ and $f_a^-=f\circ\phi_a^-$. If $g:[0,1]\to[0,\infty)$ is any function, its extension $\hat g:\RR\to[0,\infty)$ is defined as in Equation \eqref{extension}. For easier readability we sometimes omit the $+$ sign and e.g. write $f_a$ instead of $f_a^+$.

We are also going to make use of the following numbers associated to $f$:
\begin{align*}
M_a^+(f)&=M(\hat f_a^+|_{(-\infty,1]}),&M^+(f)&=\min_{a\in S^1}M_a^+(f),\\
M_a^-(f)&=M(\hat f_a^-|_{(-\infty,1]}),&M^-(f)&=\min_{a\in S^1}M_a^-(f).
\end{align*}
\end{notation}

We proceed to prove a lemma.

\begin{lemma}\label{prerez}
Suppose $f:S^1\to[0,\infty)$ has no zeros and $\ucat(f)=n\in\NN$. Then it is possible to choose $a\in S^1$ so that $\ucat(f_a)\leq\ucat(f)+1$. Furthermore, $f_a$ admits a unimodal decomposition $u_1,\ldots,u_{n+1}$ where $u_1$ is decreasing and $u_{n+1}$ is increasing.
\end{lemma}

\begin{proof}
Let $v_1,\ldots,v_n$ be a minimal unimodal decomposition of $f$ and choose $a\in S^1$ such that $v_1(x)\leq v_1(a)$ holds for all $x\in S^1$. Define functions $\tilde v_i:[0,1]\to[0,\infty)$ by $\tilde v_i=v_i\circ\phi_a$. By unimodality, each ``open support'' $\tilde v_i^{-1}(0,\infty)$ has either one or two components. Reorder the indices $i>1$ and choose $k$ so that the open supports of $\tilde v_1,\ldots,\tilde v_k$ have two components and the open supports of $\tilde v_{k+1},\ldots,\tilde v_n$ have one component.

Each $\tilde v_i$ for $1\leq i\leq k$ can uniquely be split as a sum of two unimodal functions $\tilde v_i=\tilde v_i^{(1)}+\tilde v_i^{(2)}$ so that $0\in\supp\tilde v_i^{(1)}$ and $1\in\supp\tilde v_i^{(2)}$. By unimodality, for each $i$ either $\tilde v_i^{(1)}$ is decreasing or $\tilde v_i^{(2)}$ is increasing and by our choice of $a$ both of these hold for $i=1$. After possibly reordering the indices $i\in\{2,\ldots,k\}$ again, we can choose an $l$ such that $\tilde v_i^{(1)}$ is decreasing for $i=2,...,l$ and $\tilde v_i^{(2)}$ is increasing for $i=l+1,\ldots,k$. (For $i=1$, both of these hold.)

For $i=1,\ldots,n+1$, we may therefore define $u_i:[0,1]\to[0,\infty)$ as follows:
\[
u_i=\begin{cases}\sum_{i=1}^l\tilde v_i^{(1)};&\text{for }i=1,\\
\tilde v_i^{(2)};&\text{for }i=2,\ldots,l,\\
\tilde v_i^{(1)};&\text{for }i=l+1,\ldots,k,\\
\tilde v_i;&\text{for }i=k+1,\ldots,n,\\
\tilde v_1^{(2)}+\sum_{i=l+1}^k\tilde v_i^{(2)};&\text{for }i=n+1.
\end{cases}
\]
The function $u_1$ is decreasing since it is the sum of a sequence of decreasing functions and $u_{n+1}$ is increasing since it is the sum of a sequence of increasing functions. The unimodality of the other functions is clear.
\end{proof}

\begin{proposition}\label{two}
Suppose $f:S^1\to[0,\infty)$ has no zeros and $\ucat(f_a)=2$ for some $a\in S^1$. Then $\ucat(f)=2$.
\end{proposition}

\begin{proof}
Suppose $f_a=u_1+u_2$ is the unimodal decomposition obtained by sweeping. Using Proposition \ref{modsweep}, we can modify it so that $u_1(1)=0$. If $u_1$ is increasing on some small interval $[0,\epsilon]$, let
\[
v(x)=\max\{0,(1-\tfrac x{\epsilon})f_a(0)\}
\]
and define functions $\tilde u_1,\tilde u_2:S^1\to[0,\infty)$ by $\tilde u_1\circ\phi_a=u_1-v$ and $\tilde u_1+\tilde u_2=f$. It can be directly verified that these are unimodal. An entirely analogous construction works if $u_2$ is decreasing on some interval $[1-\epsilon,1]$.

Finally, if $u_1$ is decreasing everywhere and $u_2$ is increasing, choose a point $x_0$ such that $u_1(x_0)=u_2(x_0)$ and define
\[
u(x)=\begin{cases}2u_2(x);&x\leq x_0,\\
2u_1(x);&x\geq x_0.
\end{cases}
\]
Define functions $\tilde u_1,\tilde u_2:S^1\to[0,\infty)$ by $\tilde u_1\circ\phi_a=u$ and $\tilde u_1+\tilde u_2=f$. A direct verification shows that these are unimodal.
\end{proof}

This allows us to characterize $\ucat(f)$ in the following way.

\begin{theorem}\label{minindex}
The unimodal category of $f:S^1\to[0,\infty)$ without zeros and of bounded variation is characterized as follows:
\[
\ucat(f)=\max\{2,M^+(f)\}=\max\{2,M^-(f)\}.
\]
\end{theorem}

\begin{proof}
Clearly, $\ucat(f)\geq\gcat(S^1)=2$, since $f$ has no zeros. Notice that is suffices to prove the first equality. The second equality then follows by a simple observation that composing $f$ with a reflection of the circle does not affect the unimodal category, since a reflection is a homeomorphism.

Next, we show that $M^+(f)\leq\ucat(f)=:n$. By Lemma \ref{prerez}, there is a point $a\in S^1$ and a unimodal decomposition $f_a=\sum_{i=1}^{n+1}u_i$ such that $u_{n+1}$ is increasing. By Proposition \ref{intervals} this means that $M_a^+(f)\leq n$, which proves $M^+(f)\leq\ucat(f)$.

Note that this already completes the proof in the case $\ucat(f)=2$. We may therefore assume that $n=\ucat(f)\geq 3$, in which case it remains to establish that $\ucat(f)\leq M^+(f)$. By the definition of $M^+(f)$, it suffices to show that $\ucat(f)\leq M_a^+(f)$ holds for all $a\in S^1$. So let $a\in S^1$. By Proposition \ref{intervals}, $M_a^+(f)$ is either $\ucat(f_a)$ or $\ucat(f_a)-1$.

{\bf Case 1:} $M_a^+(f)=\ucat(f_a)$. In this case, it is sufficient to see that $\ucat(f_a)\geq n$. Suppose that this does not hold. Then the sweeping algorithm returns a unimodal decomposition of $f_a$ with $k<n$ summands and by Proposition \ref{modsweep} we may modify the last two so that $u_i(1)=0$ for $i<k$, which also implies $u_k(1)=u_1(0)$. If $k=2$, Proposition \ref{two} yields a contradiction. Otherwise note that for $i=2,\ldots,k-1$, we have $u_i^{-1}(0,\infty)\subseteq(0,1)$, so these summands may be transported to $S^1$ by defining $\tilde u_i\circ\phi_a=u_i$. Since $k>2$, the sets $u_1^{-1}(0,\infty)$ and $u_k^{-1}(0,\infty)$ are disjoint. The summands $u_1$ and $u_k$ may be glued together, i.e. we may define $g:S^1\to[0,\infty)$ by $g\circ\phi_a=u_1+u_k$. If $g$ is unimodal, we have $\ucat(f)\leq k-1$, a contradiction. If not, there exist $B>A>0$ such that for $x>B$, $g^{-1}[x,\infty)$ is empty, for $x\in(A,B]$ it has two contractible components and for $x\in(0,A]$ it has one contractible component. This implies that $\ucat(g)=2$ and therefore $\ucat(f)\leq k$, which is again a contradiction.

{\bf Case 2:} $M_a^+(f)=\ucat(f_a)-1$. By Proposition \ref{intervals}, this means that $f_a$ has a unimodal decomposition into $k=\ucat(f_a)$ summands $u_1,\ldots,u_k$ such that $u_k$ is increasing. It is sufficient to prove that $k\geq n+1$. Suppose not: then $k<n+1$. If $k=2$, we have $\ucat(f)=2$ by Proposition \ref{two}, so we may assume that $k\geq 3$. By Proposition \ref{modsweep}, we may again modify the last two summands so that $u_{k-1}(1)=0$. This allows us to use the same construction as in Case 1: define $\tilde u_i$ by $\tilde u_i\circ\phi_a=u_i$ for $i=2,\ldots,k-1$ and $\tilde u_1\circ\phi_a=u_1+u_k$. Since $u_k$ is increasing, $\tilde u_1$ is unimodal by definition. But this implies that $\ucat(f)\leq k-1<n$, yet another contradiction.
\end{proof}

\begin{remark}
As a consequence of this proof we can see that $M_a^+(f)$ can be computed in a greedy manner as in Theorem \ref{sweeping}, i.e. by recursively defining
\begin{align*}
x_0&=-\infty,\\
x_i&=\inf\{x\mid V^-(\hat f_a;(x_{i-1},x))>\hat f_a(x_{i-1})\},\qquad i\in\NN,
\end{align*}
however, in this case, we must stop once $x_i$ exceeds $1$.
\end{remark}

\subsubsection{Algorithm in the Case of Finitely Many Critical Points}

One possible interpretation of these results is that $\ucat(f)$ of $f:S^1\to[0,\infty)$ can be computed by sweeping if we know where to start. In general, it is not entirely clear how to find the starting point. However, if the function only has finitely many critical points, it suffices to check these critical points. Therefore the unimodal category of any such function can be computed in a completely algorithmic manner. This is justified by the following result which shows that $M_a^+(f)$ achieves its minimum at a critical point. An explicit description of the algorithm is available in Appendix \ref{algorithms}.

For convenience we introduce some further notation. If $a,a'\in S^1$ let $[a,a']$ denote the closed arc between $a$ and $a'$, i.e. the set of points obtained by starting at $a$ and traversing the circle in the positive direction until reaching $a'$. Let $(a,a'),(a,a']$ and $[a,a')$ denote the corresponding open and half-open arcs.

\begin{proposition}
Suppose $f:S^1\to[0,\infty)$ is a function and $a_1,a_2\in S^1$ are points such that $(a_1,a_2)$ does not contain a critical point of $f$. Then
\[
M_{a_1}^+(f)\leq M_a^+(f)
\]
for all $a\in(a_1,a_2)$.
\end{proposition}

\begin{proof}
Let $0\leq t_1<t<t_2\leq 1$, $a_1=\exp(2\pi i t_1)$, $a_2=\exp(2\pi i t_2)$ and $a=\exp(2\pi i t)$. Also let $t'=t-t_1$. Suppose $M_{a_1}^+(f)=n$ and let $(-\infty,x_1),(x_1,x_2),\ldots,(x_{n-1},x_n)$ be a collection of forced-max intervals for $\hat f_{a_1}|_{(-\infty,1]}$. There are now two cases to treat.

{\bf Case 1:} $\hat f_{a_1}$ is increasing on $[0,t']$. In this case, we have $x_1\geq t'$ and the intervals $(-\infty,x_1-t'),(x_1-t',x_2-t'),\ldots,(x_{n-1}-t',x_n-t')$ is obviously a collection of forced-max intervals for $\hat f_a|_{(-\infty,1]}$, so $M_a^+(f)\geq n$.

{\bf Case 2:} $\hat f_{a_1}$ is decreasing on $[0,t']$. Without loss of generality, we can assume $\hat f_{a_1}$ is non-constant on $[0,t']$, otherwise the first case applies. We also have $t'\leq x_2$. Without loss of generality we can further assume that $x_1\leq t'$. (If this is not the case, we can replace $x_1$ by $x:=\min\{x_1,t'\}$. Upon doing this, $(-\infty,x)$ is still forced-max, because $\hat f_{a_1}$ is decreasing on $[0,t']$, and $(x,x_2)$ is forced-max because forced-max intervals form an ideal.) There is an $\epsilon>0$ such that $\hat f_{a_1}$ is still decreasing on $[0,t'+\epsilon]$. The intervals $(-\infty,\epsilon),(\epsilon,x_2-t'),(x_2-t',x_3-t'),\ldots,(x_{n-1}-t',x_n-t')$ are a collection of forced-max intervals for $\hat f_a|_{(-\infty,1]}$, so $M_a^+(f)\geq n$. For $(\epsilon,x_2-t')$, this is true because the interval $(t'+\epsilon,x_2)$ is still forced-max for $\hat f_{a_1}|_{(-\infty,1]}$ since it is decreasing on $[x_1,t'+\epsilon]$ and therefore
\[
V^-(\hat f_a;(\epsilon,x_2-t'))=V^-(\hat f_{a_1};(t'+\epsilon,x_2))>\hat f_{a_1}(t'+\epsilon)=\hat f_a(\epsilon).
\]
For all the other intervals, this is obvious.
\end{proof}

\begin{example}
Consider the function $f:S^1\to[0,\infty)$ obtained by choosing eight points on the circle, for instance $a_j=\exp(2\pi i t_j)\in S^1$, $j=0,1,\ldots,7$, where $0=t_0<t_1<\ldots<t_7<1$, taking the values there to be $4,3,\frac72,3,4,1,3,1$ respectively, and interpolating linearly in between. Slicing the circle at $a=0$, we obtain the following graph:
\begin{center}
\includegraphics[width=340pt]{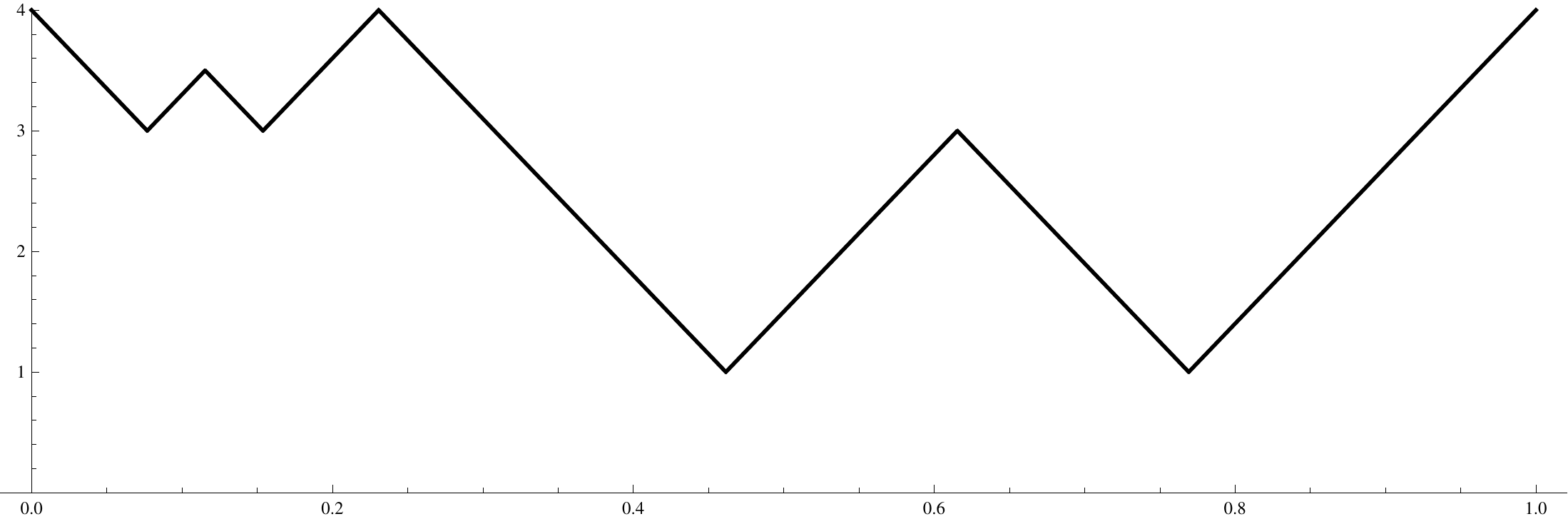}
\end{center}
By the above proposition, the calculation of $M_{a_i}^+(f)$ for $i=0,1,\ldots,7$ suffices to determine $\ucat(f)$. However, it is illustrative to calculate $M_a^+(f)$ for all $a\in S^1$. By a direct calculation, we obtain:
\[
M_a^+(f)=\begin{cases}2;&a\in[a_1,a_2]\cup[a_5,a_6],\\
3;&a\in(a_6,a_1)\cup(a_2,a_5).
\end{cases}
\]
We conclude that $\ucat(f)=2$. An explicit unimodal decomposition can also be described. The first summand is obtained by mapping the points $a_j$, $j=0,1,\ldots,7$, to the sequence $3,3,\frac72,3,3,0,0,0$, and interpolating linearly. The second summand is obtained by mapping the points $a_j$, $j=0,1,\ldots,7$, to the sequence $1,0,0,0,1,1,3,1$ and interpolating linearly.
\end{example}

\section{The Monotonicity Conjecture}\label{monotonicity}

Baryshnikov and Ghrist conclude their paper with the following conjecture \cite[Conjecture 18]{baryshnikov}, which they believe could play an important role in obtaining various bounds for $\ucat^p$.

\begin{conjecture}
Suppose $f:X\to[0,\infty)$ and $0<p_1<p_2\leq\infty$. Then $\ucat^{p_1}(f)\leq\ucat^{p_2}(f)$.
\end{conjecture}

In other words, they conjecture that $\ucat^p$ is monotone in $p$. The aim of this section is to investigate this conjecture for various spaces and functions. We show that the conjecture is true for $X=\RR$ and $X=S^1$. However, there are counterexamples if $X$ is a certain type of graph and if $X=\RR^2$ is the Euclidean plane.

\subsection{Proof for the Real Line and the Circle}\label{real_monotonicity}

In the case of $X=\RR$, the conjecture is true.
\begin{definition}
Let $k\in\NN_0$. We call a sequence $a=(a_0,a_1,\ldots,a_{2k})$ such that $a_0\leq a_1\geq a_2\leq\ldots\geq a_{2k}$ an {\bf up-down} sequence of length $2k$. For such a sequence, define its {\bf negative variation} as follows:
\[
V^-(a)=\sum_{i=1}^{k}(a_{2i-1}-a_{2i}).
\]
Note that this is a sum of non-negative numbers. We also define its $p$-th power for $p>0$ as $a^p=(a_0^p,a_1^p,\ldots,a_{2k}^p)$. Note that this is again an up-down sequence.
\end{definition}

Now, recall the following inequality of Karamata:
\begin{theorem}[\cite{kadelburg}]\label{karamata}
Suppose we are given two finite sequences $a=(a_1,\ldots,a_n)$ and $b=(b_1,\ldots,b_n)$ with terms in $(\alpha,\beta)$ satisfying the following conditions\footnote{If these conditions are satisfied, we say that $a$ majorizes $b$.}:
\begin{itemize}
\item $a_1\geq a_2\geq\ldots\geq a_n$ and $b_1\geq b_2\geq\ldots\geq b_n$,
\item $a_1+a_2+\ldots+a_k\geq b_1+b_2+\ldots+b_k$ for $1\leq k\leq n-1$,
\item $a_1+a_2+\ldots+a_n=b_1+b_2+\ldots+b_n$.
\end{itemize}
Further suppose $\phi:(\alpha,\beta)\to\RR$ is a convex function. Then we have
\[
\phi(a_1)+\phi(a_2)+\ldots+\phi(a_n)\geq\phi(b_1)+\phi(b_2)+\ldots+\phi(b_n).
\]
\end{theorem}

Note that the inequality is reversed if $\phi$ is concave, because in that case $-\phi$ is convex. We only need the following special case:

\begin{lemma}
Suppose $0<q<1$ and let $x,y,z\geq 0$ such that $\max\{x,z\}\leq y\leq x+z$. Then
\[
(x-y+z)^q\leq x^q-y^q+z^q.
\]
\end{lemma}

\begin{proof}
By continuity, it is sufficient to treat the case when $x,y,z>0$ and $y<x+z$. The inequality is equivalent to
\[
(x-y+z)^q+y^q\leq x^q+z^q,
\]
which is just a special case of Theorem \ref{karamata} in the case of $n=2$, $a_1=y,a_2=x-y+z$, $b_1=\max\{x,z\},b_2=\min\{x,z\}$ and $\phi(t)=t^q$, which is concave.
\end{proof}

\begin{lemma}\label{updown}
Suppose $0<q<1$, $a$ is an up-down sequence of length $2k$ and $V^-(a)\leq a_0$. Then $V^-(a^q)\leq a_0^q$.
\end{lemma}

\begin{proof}
We prove this by induction on $k$. For $k=0$, this is trivial, as both negative variations are zero. Suppose the lemma holds for all sequences of length at most $2k-2$ and $a$ is a sequence of length $2k$. Then $\tilde a=(a_0-a_1+a_2,a_3,a_4,\ldots,a_{2k-1},a_{2k})$ is an up-down sequence of length $2k-2$. Observe that
\[
V^-(\tilde a)=V^-(a)-(a_1-a_2)\leq a_0 - a_1 + a_2.
\]
By the inductive hypothesis and the previous lemma, we have
\[
V^-(\tilde a^q)\leq (a_0 - a_1 + a_2)^q \leq a_0^q - a_1^q + a_2^q.
\]
But
\[
V^-(\tilde a^q)+a_1^q-a_2^q = V^-(a^q),
\]
so this is precisely the conclusion we wanted.
\end{proof}

We can now proceed to the proof of the monotonicity conjecture. First, we note that there is a typo in the statement Lemma 9 in \cite{baryshnikov}, which relates the various notions of $\ucat^p$ for $p<\infty$. The proof as stated there, remains valid. For convenience, we restate the lemma in its correct form.

\begin{lemma}[\cite{baryshnikov}, Lemma 9]\label{powers}
If $f:X\to[0,\infty)$ is any continuous function, then
\[
\ucat^p(f)=\ucat(f^p).
\]
\end{lemma}

We can now prove the monotonicity conjecture for functions $f:\RR\to[0,\infty)$.

\begin{theorem}\label{monotonicity_real}
Suppose $f:\RR\to[0,\infty)$ is of bounded variation and $0<p_1<p_2\leq\infty$. Then $\ucat^{p_1}(f)\leq\ucat^{p_2}(f)$.
\end{theorem}

\begin{proof}
In the case of $p_2=\infty$ this follows trivially from Lemma 6, Lemma 9 (our Lemma \ref{powers}) and Lemma 16 of \cite{baryshnikov}, so it suffices to treat the case where $p_1,p_2<\infty$. In this case, observe that $\ucat^{p_1}(f)\leq\ucat^{p_2}(f)$ is equivalent to $\ucat(f^{p_1})\leq\ucat^{\frac{p_2}{p_1}}(f^{p_1})$ by Lemma \ref{powers}. Therefore, we may assume without loss of generality that $p_1=1$ and $p_2=p>1$. Under this reduction, we want to prove that $\ucat(f)\leq\ucat(f^p)$.

To prove this, observe that it is sufficient to prove the following statement:
\[
\text{if $V^-(f;[x,y])>f(x)$, then $V^-(f^p;[x,y])>f(x)^p$.}
\]
This means precisely that each forced-max interval of $f$ is also a forced interval of $f^p$, which then establishes the claim by Theorem \ref{sweeping}. We are going to prove the contrapositive. Assume therefore that $V^-(f^p;[x,y])\leq f(x)^p$. By the definition of negative variation, this means precisely that
\[
\sup\left[\sum_{i=1}^n\max\{0,f(a_{i-1})^p-f(a_i)^p\}\right]\leq f(x)^p,
\]
where the supremum is taken over all partitions $P$ of the form $x=a_0\leq a_1\leq a_2\leq\ldots\leq a_n=y$ of the interval $[x,y]$. Fix an arbitrary partition $P$ of this kind. Note that this time, we allow duplications $a_i=a_{i+1}$ in $P$. This does not change the relevant supremum and allows us to assume that $n=2k$ for some $k\in\NN$ and that
\[
f(a_0)\leq f(a_1)\geq f(a_2)\leq\ldots\geq f(a_{2k}).
\]
Let $c$ be the up-down sequence of length $2k$ defined by $c_i=f(a_i)^q$. We have
\[
V^-(c)=\sum_{i=1}^n\max\{0,f(a_{i-1})^p-f(a_i)^p\}\leq f(x)^p=c_0.
\]
By Lemma \ref{updown} with $q=\frac1p$, we have
\[
V^-\left(c^{\frac1p}\right)\leq c_0^{\frac1p}.
\]
This means precisely that
\[
\sum_{i=1}^n\max\{0,f(a_{i-1})-f(a_i)\}\leq f(x).
\]
Since the partition $P$ was arbitrary, the same holds for the supremum over all partitions:
\[
V^-(f;[x,y])\leq f(x).
\]
This concludes the proof.
\end{proof}

This immediately proves the monotonocity conjecture for the circle $S^1$.

\begin{corollary}
Suppose $f:S^1\to[0,\infty)$ is of bounded variation and $0<p_1<p_2\leq\infty$. Then $\ucat^{p_1}(f)\leq\ucat^{p_2}(f)$.
\end{corollary}

\begin{proof}
In the case of $p_2=\infty$ this follows from Lemma 6, Lemma 9 (our Lemma \ref{powers}) and Lemma 16 of \cite{baryshnikov} (note that Lemma 16 still holds for $S^1$), so it suffices to treat the case $p_1<p_2<\infty$. In this case, we use Lemma \ref{powers} and Theorem \ref{minindex}:
\[
\ucat^{p_1}(f)=\min\{2,M^+(f^{p_1})\}\leq\min\{2,M^+(f^{p_2})\}=\ucat^{p_2}(f).
\]
The inequality follows from Theorem \ref{monotonicity_real}.
\end{proof}

\subsection{Graphs: First Counterexample}\label{graphs1}

For general spaces, the monotonicity conjecture is false. The simplest counterexamples can be constructed on graphs. Similar ideas can then be exploited to yield counterexamples on Euclidean spaces.

Let $G$ be the graph (abstract simplicial complex of dimension $1$) whose vertices and edges are given by\footnote{Here, $xy$ is considered as shorthand for $\{x,y\}$.}
\begin{align*}
V&=\{a_1,a_2,b_1,b_2,c,d_1,d_2,d_3,e_1,e_2,e_3,q\},\\
E&=\{a_1a_2,a_2c,b_1b_2,b_2c,cd_1,ce_1,d_1d_2,d_2d_3,d_3f,e_1e_2,e_2e_3,e_3q\},
\end{align*}
and let $X$ be the polytope of its geometric realization, for instance as in the picture below, where the vertices $c$ and $q$ are realized as $(0,0)$ and $(-2,-2)$, $a_1,a_2,d_1,d_2,d_3$ are realized as $(2,0), (1,0), (-1,0), (-2,0)$ and $(-2,-1)$ and $b_1,b_2$, $e_1,e_2,e_3$ as their reflections across $y=x$. For simplicity, we identify the vertices of $G$ with their corresponding points in $X$.

\begin{center}
\includegraphics[width=170pt]{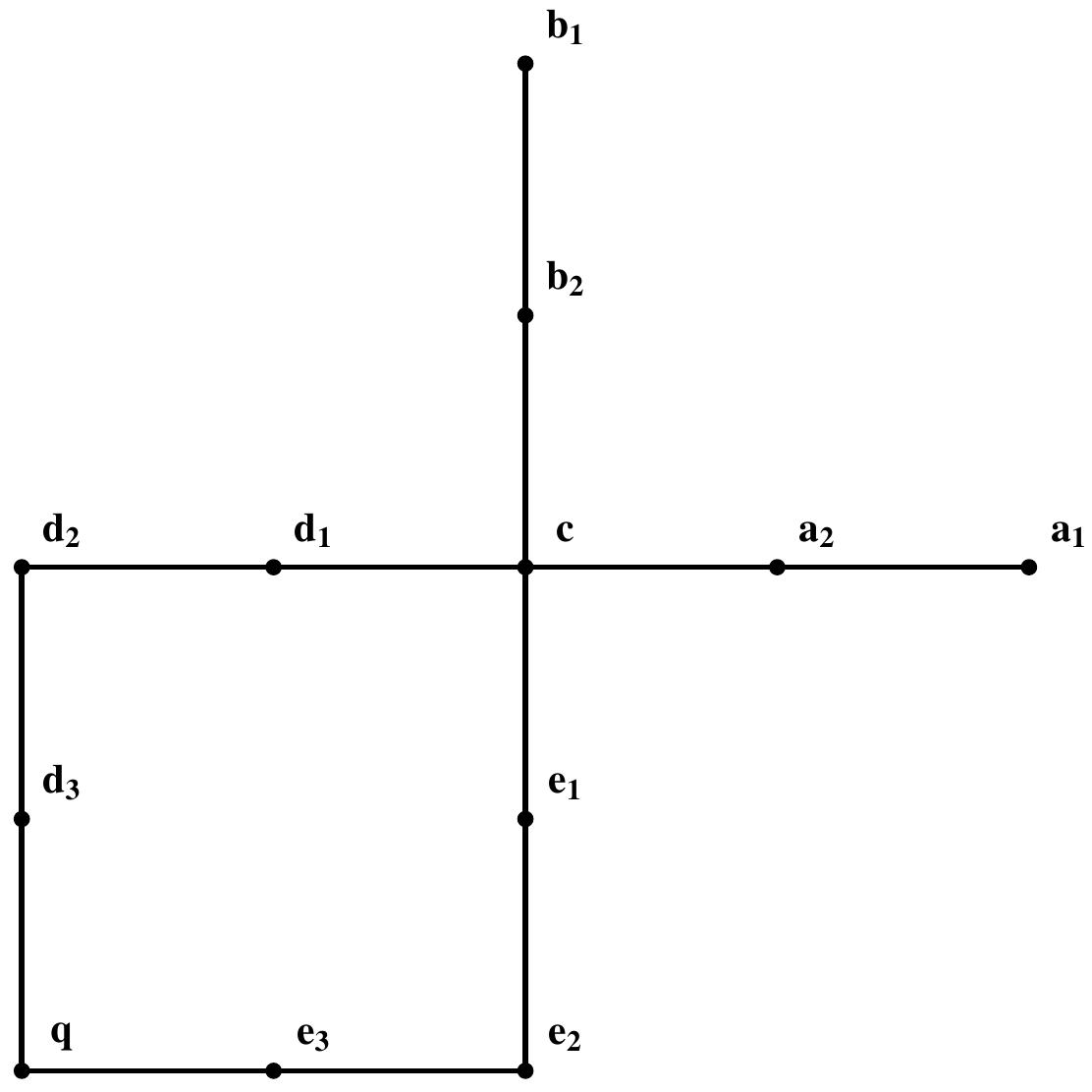}
\end{center}

We define three piecewise linear (with respect to the graph structure above) functions $f,u_1,u_2:X\to[0,\infty)$. The functions $u_1$ and $u_2$ are defined by specifying their values on the vertices ($i=1,2,3$):
\begin{align*}
u_1(a_1)&=5,& u_1(a_2)&=u_1(c)=u_1(d_i)=u_1(q)=1,& u_1(b_1)&=u_1(b_2)=u_1(e_i)=0,\\
u_2(b_1)&=5,& u_2(b_2)&=u_2(c)=u_2(e_i)=u_2(q)=1,& u_2(a_1)&=u_2(a_2)=u_2(d_i)=0.
\end{align*}
See picture (where $u_2$ is $u_1$ reflected across the axis of symmetry):

\begin{center}
\includegraphics[width=170pt]{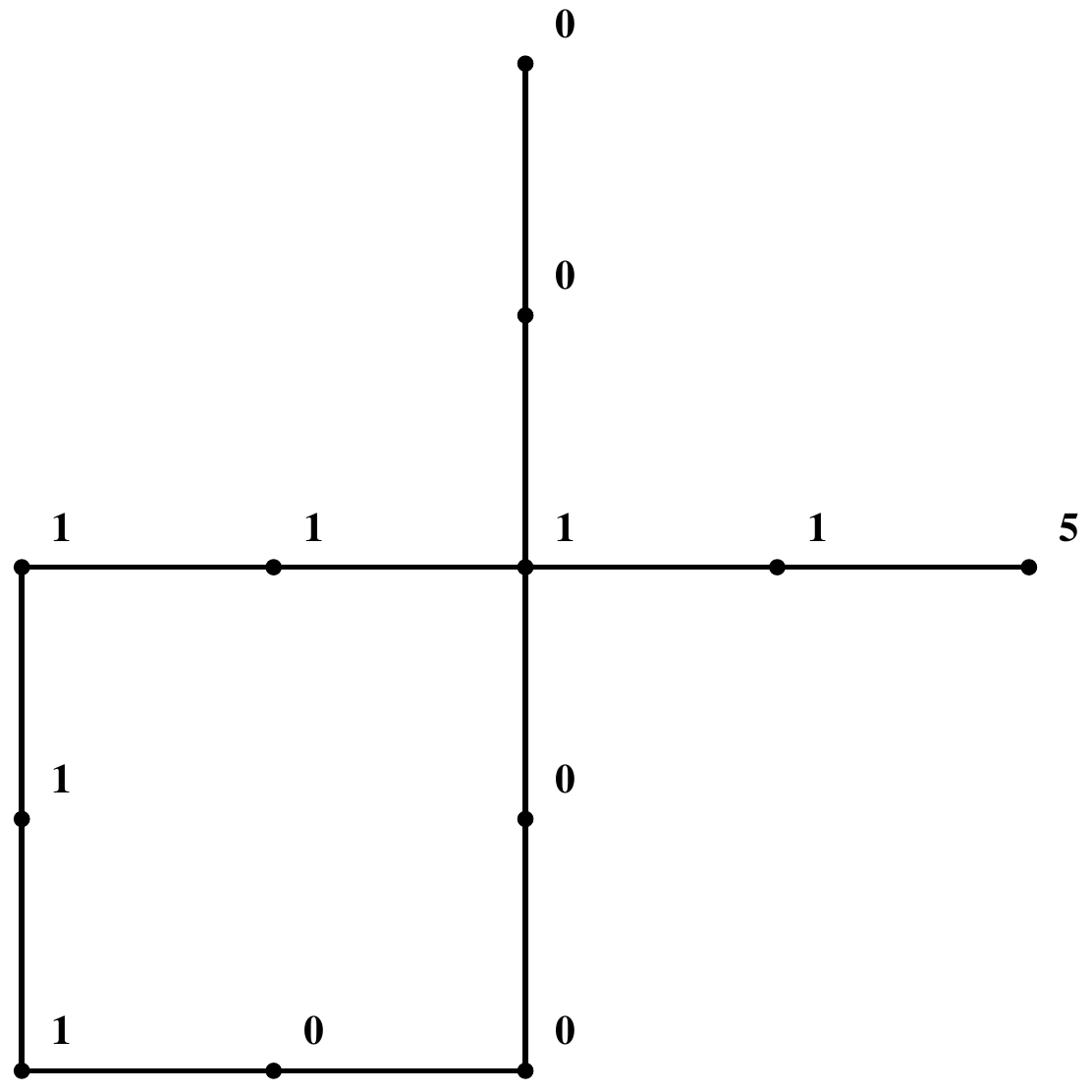} \includegraphics[width=170pt]{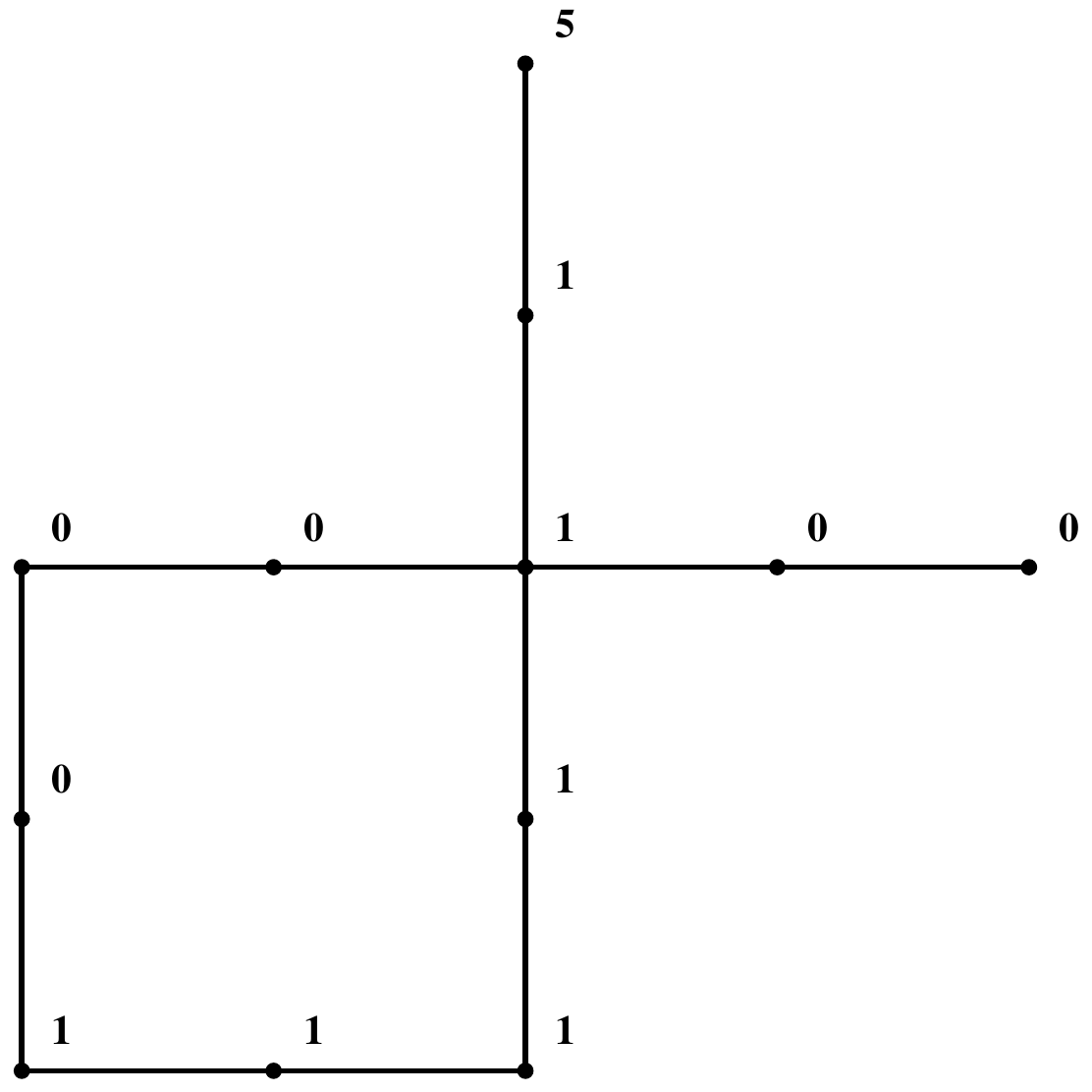}
\end{center}

Finally, we define $f=u_1+u_2$. Note that $f$ may also be given by its values at the vertices ($i=1,2,3$):
\begin{align*}
f(a_1)&=f(b_1)=5,& f(c)&=f(q)=2,& f(a_2)=f(b_2)=f(d_i)=f(e_i)=1.
\end{align*}

\begin{center}
\includegraphics[width=170pt]{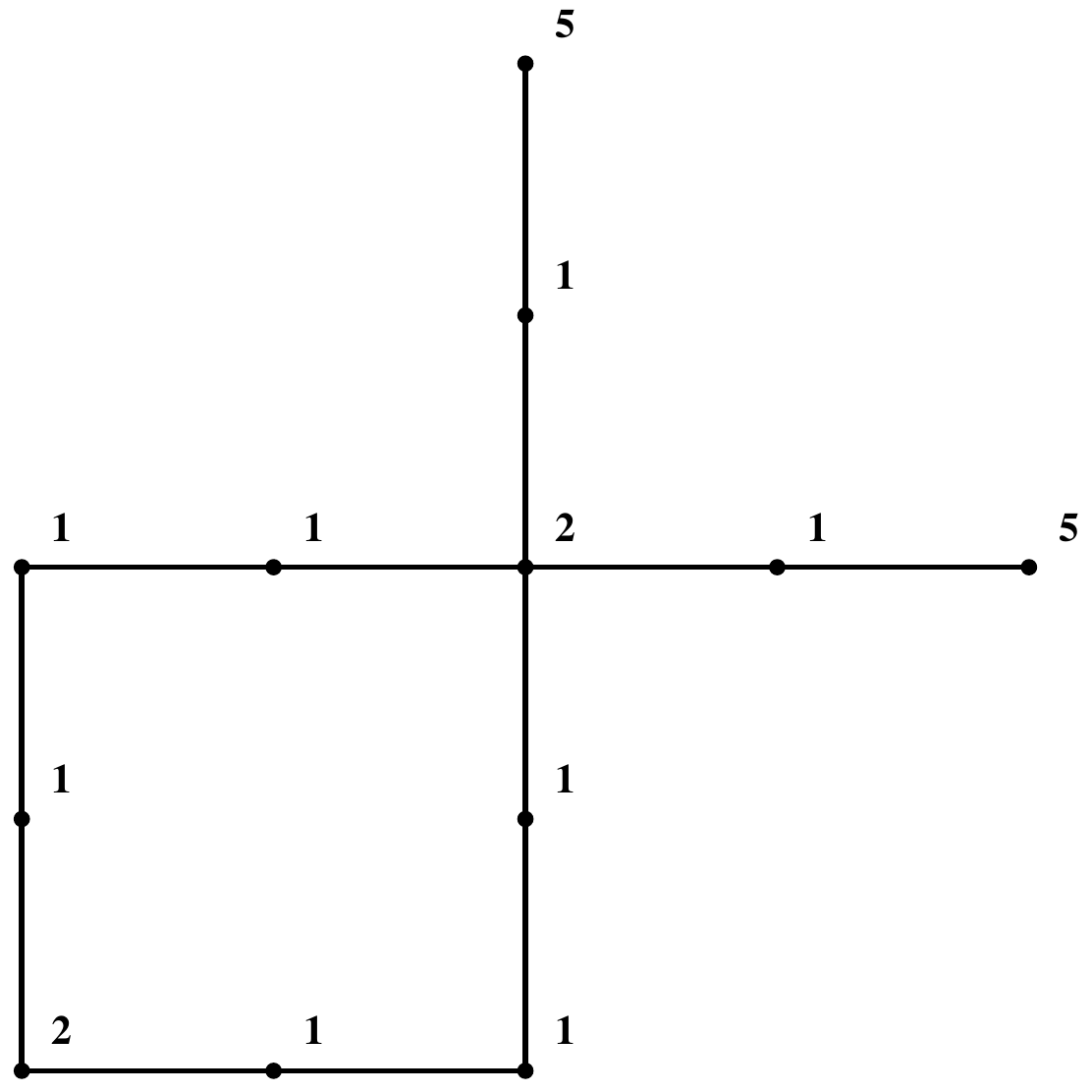}
\end{center}

The following follows directly from the definitions.

\begin{observation}
The functions $u_1$ and $u_2$ are unimodal, $f$ is not. So $\ucat(f)=2$.
\end{observation}

This $f$ yields a simple counterexample to the monotonicity conjecture.

\begin{proposition}
The unimodal $\frac12$-category of $f$ is $\ucat(\sqrt{f})=\ucat^{\frac12}(f)=3$.
\end{proposition}

\begin{proof}
A unimodal decomposition of length $3$ can easily be constructed explicitly. Note that $\sqrt{f}$ is not piecewise linear, but can be turned into such by using an appropriate homeomorphism on the domain $X$, all the while retaining the function values at the vertices. The resulting piecewise linear function can be decomposed into three piecewise linear unimodal summands $v_1,v_2$ and $v_3$, which we define by their values at the vertices. The nonzero values are given by ($i=1,2,3$):
\begin{align*}
v_1(a_1)&=\sqrt{5},& v_1(a_2)&=1,& v_1(c)&=\tfrac{\sqrt{2}}{2},\\
v_2(a_1)&=\sqrt{5},& v_2(a_2)&=1,& v_2(c)&=\tfrac{\sqrt{2}}{2},\\
v_3(q)&=\sqrt{2},& v_3(d_i)&=1,& v_3(e_i)&=1.
\end{align*}
The function values at the remaining vertices are zero.

To complete the proof, it therefore remains to show that $\sqrt{f}$ cannot be decomposed into two unimodal summands. We argue by contradiction. Suppose $\sqrt{f}=v_a+v_b$, where $v_a$ and $v_b$ are unimodal. If $r>0$, define superlevel sets
\[
Q_a(r)=v_a^{-1}[r,\infty)\qquad\text{and}\qquad Q_b(r)=v_b^{-1}[r,\infty).
\]
By unimodality, these are all contractible. Note that, since $v_a(a_1)+v_b(a_1)=\sqrt{f(a_1)}$, we either have $v_a(a_1)\geq\frac{\sqrt{5}}2$ or $v_b(a_1)\geq\frac{\sqrt{5}}2$. Without loss of generality, assume that the first possibility holds. (This is also the reason behind the choice of notation for $v_a$ and $v_b$.) Since $\sqrt{f(a_2)}=1$, we have $v_a(a_2)\leq 1$. This immediately implies that $v_a(b_1)\leq 1$, since otherwise we would have $v_a(b_1)=r>1$ and $Q_a(\min\{r,\frac{\sqrt{5}}2\})$ would not be connected (since it is a subspace of $X$ containing $a_1$ and $b_1$ but not $a_2$). This means that $v_b(b_1)\geq\sqrt5-1$. By a symmetric argument, we also have $v_a(a_1)\geq\sqrt5-1$, but we do not use this fact.

Now, let $K$ be the subspace of $X$ consisting of the vertices $c,d_i,e_i,q$ ($i=1,2,3$) and all the edges between these vertices. Note that $K$ is homeomorphic to a circle. By unimodality, there are points $x,y\in K$ such that $v_a(y)=0$ and $v_b(x)=0$. Otherwise we would have $K\subseteq Q_a(r)$ or $K\subseteq Q_b(r)$ for some $r>0$. Therefore $v_a(x)=v_b(y)=1$. Since $Q_a(1)$ and $Q_b(1)$ are contractible, there is a path from $a_1$ to $x$ in $Q_a(1)$, implying that $c\in Q_a(1)$, and a path from $b_1$ to $y$ in $Q_b(1)$, implying that $c\in Q_b(1)$. This implies that $v_a(c)+v_b(c)\geq 2$, contradicting the fact that $v_a(c)+v_b(c)=\sqrt{f(c)}=\sqrt2$ and concluding the proof.
\end{proof}

To sum up, we have found a space $X$, a function $f:X\to[0,\infty)$ and values $0<p_1<p_2<\infty$ such that $\ucat^{p_1}(f)>\ucat^{p_2}(f)$. Hence, the monotonicity conjecture is not generally true.

\begin{remark}
In fact, by changing the function values of $u_1$ and $u_2$ appropriately, the same example can be modified to show that the monotonicity conjecture is not generally true for any pair of exponents $0<p_1<p_2<\infty$.

Further note that this example implies the failure of monotonicity for a very general class of graphs: namely, whenever the graph contains a cycle which contains a point of valence $4$, monotonicity cannot hold in general. In fact, monotonicity fails for an even larger class of graphs: the point of valence $4$ can be replaced by two points of valence $3$ as in the picture below, yielding another counterexample. The proof is very similar to the one above, so we omit it. So, if a connected graph contains a cycle and a point of valence $3$ or more somewhere outside this cycle, monotonicity cannot hold in general. Note that this severely limits the collection of CW complexes for which monotonicity can possibly hold.
\end{remark}

\begin{center}
\includegraphics[width=170pt]{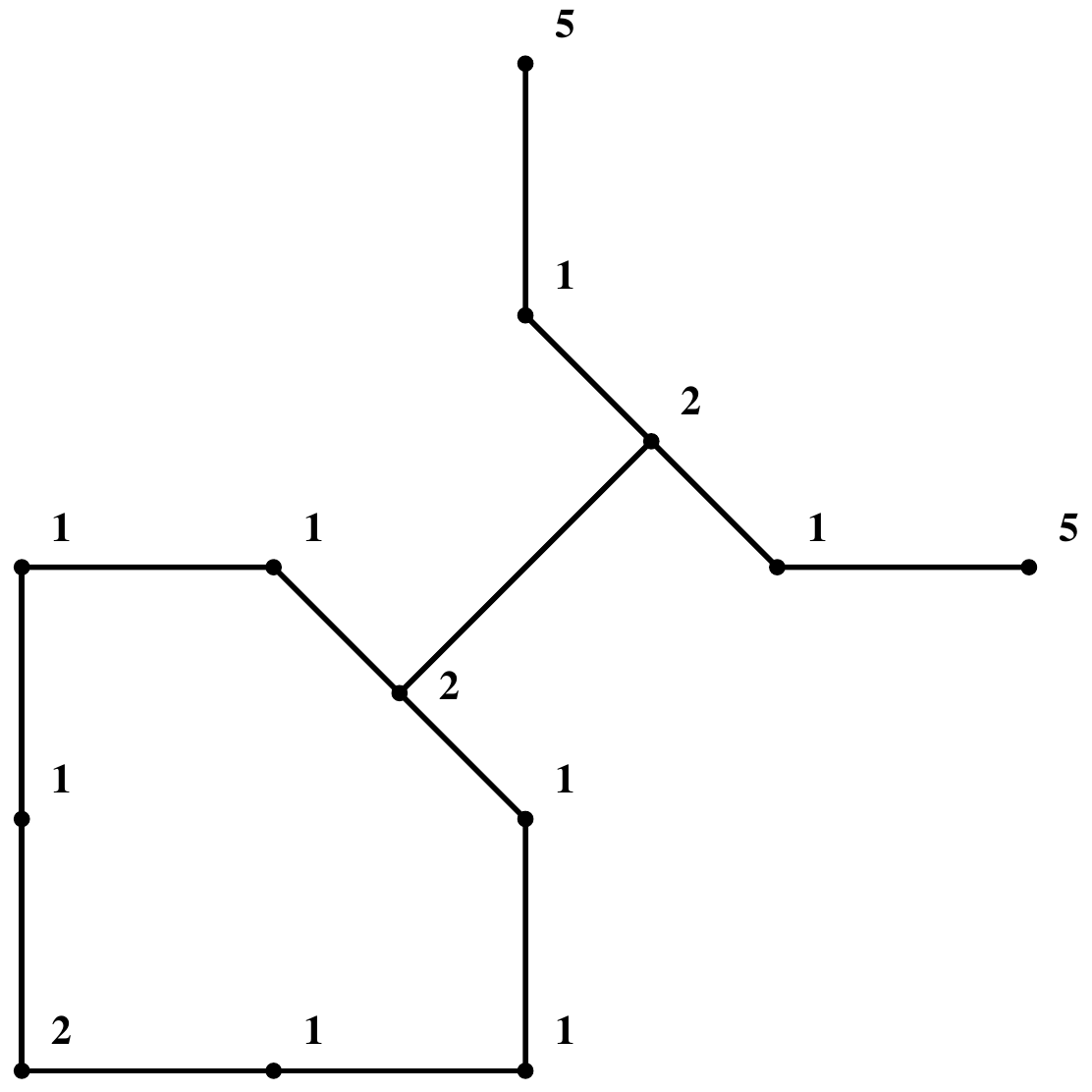}
\end{center}

\subsection{Graphs: Second Counterexample}\label{graphs2}

The previous example does not preclude the possibility that the monotonicity conjecture holds in the case $0<p_1<p_2=\infty$. We must therefore construct a different example to show that it also fails here. Let $G$ be the graph whose vertices and edges are given by
\begin{align*}
V&=\{a,b,c,d,e\},\\
E&=\{ab,ac,ad,ae,bc,bd,be,cd,ce\}.
\end{align*}
We can realize $G$ geometrically in $\RR^3$ as the $1$-skeleton of a triangular bipyramid. Let $X$ be the polytope of its geometric realization. At the cost of losing some symmetry, but simplifying the illustrations, we prefer to picture $X$ as embedded into the plane $\RR^2$:

\begin{center}
\includegraphics[width=170pt]{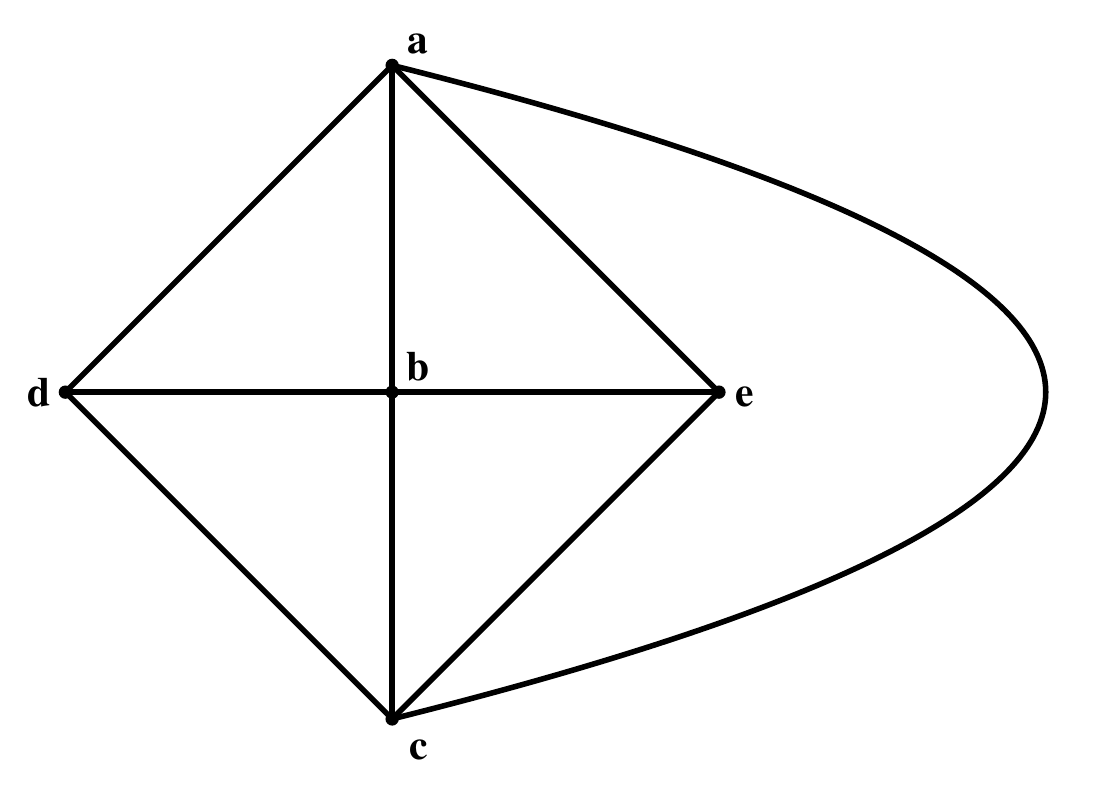}
\end{center}

We define a piecewise linear function $f:X\to[0,\infty)$ by specifying its values at the vertices:
\[
f(d)=f(e)=3\qquad\text{and}\qquad f(a)=f(b)=f(c)=1.
\]

\begin{center}
\includegraphics[width=170pt]{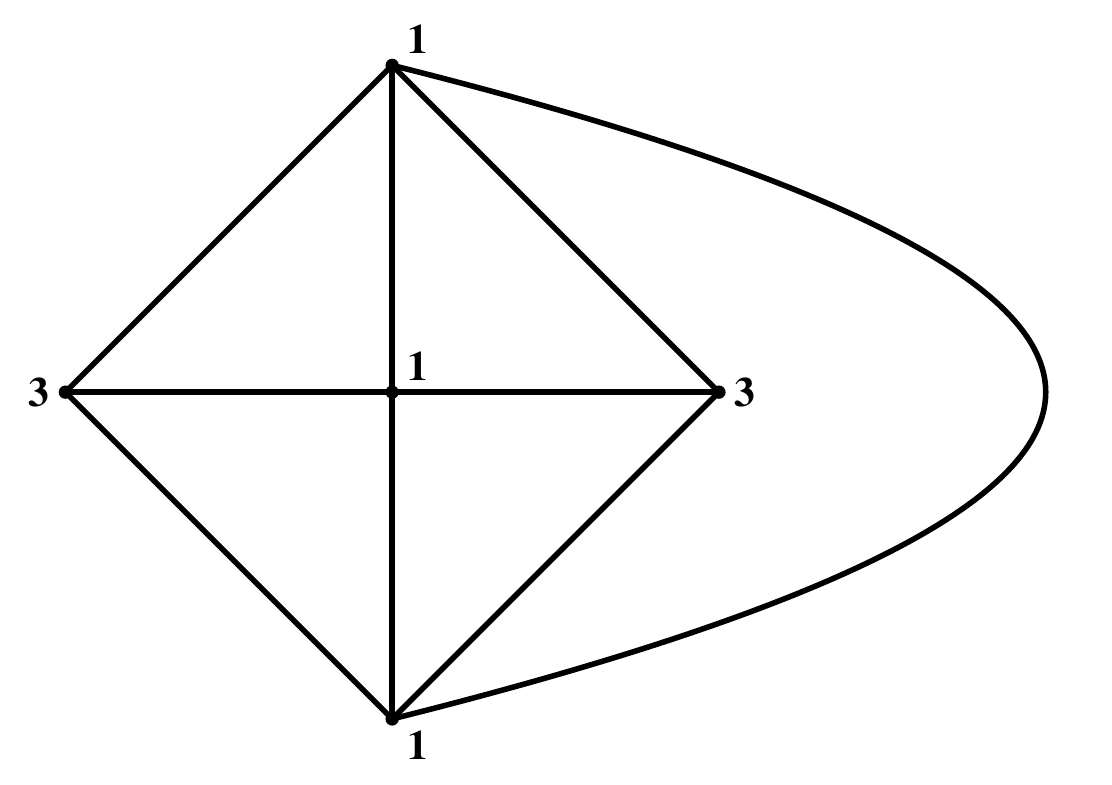}
\end{center}

We now calculate $\ucat$ and $\ucat^{\infty}$ for this function.

\begin{proposition}
The unimodal $\infty$-category of $f$ is $\ucat^{\infty}(f)=2$.
\end{proposition}

\begin{proof}
Clearly, $f$ is not unimodal. However, it has a unimodal $\infty$-decomposition of length $2$. To construct it, we first subdivide the edges $ab,ac$ and $bc$, by adding three points on each of them. Let $i,j,k$ be the vertices added on $ab$, so that it is now replaced by four edges $ai,ij,jk,kb$. Similarly, add vertices $p,q,r$ on $bc$ to subdivide it into $bp,pq,qr,rc$. Finally, add $x,y,z$ on the edge $ac$ to subdivide it into $cx,xy,yz,za$.

Now the decomposition $u_1,u_2$ can be defined by piecewise linear functions defined by the values on the vertices of this subdivision. Namely, take
\begin{align*}
&u_1(d)=3,\qquad u_1|_{\{a,b,c,j,k,q,r,y,z\}}\equiv1\qquad\text{and}\qquad u_1|_{\{e,i,p,x\}}\equiv0,\\
&u_2(e)=3,\qquad u_2|_{\{a,b,c,i,j,p,q,x,y\}}\equiv1\qquad\text{and}\qquad u_2|_{\{d,k,r,z\}}\equiv0.
\end{align*}
These have the desired properties: they are unimodal and $f=\max\{u_1,u_2\}$.
\end{proof}

\begin{proposition}
The unimodal category of $f$ is $\ucat(f)=3$.
\end{proposition}

\begin{proof}
A unimodal decomposition of length $3$ can easily be constructed explicitly so this is left to the reader.

It remains to show that there is no unimodal decomposition of length $2$. Again, we argue by contradiction. Suppose $f=u_1+u_2$ is such a decomposition. If $r>0$, define superlevel sets
\[
Q_1(r)=u_1^{-1}[r,\infty)\qquad\text{and}\qquad Q_2(r)=u_2^{-1}[r,\infty).
\]
Without loss of generality, we may assume that $u_1(d)\geq u_2(d)$, in other words, $u_1(d)\geq\frac32$. By unimodality, it follows that $u_1(e)\leq 1$ (otherwise $u_1(e)=r>1$ and the points $d$ and $e$ would have to lie in separate components of $Q_1(\min\{\frac32,r\})$). It follows that $u_2(e)\geq 2$. Using unimodality again, we have $u_2(d)\leq 1$ (otherwise $u_2(d)=r>1$ and the points $d$ and $e$ would have to lie in separate components of $Q_2(\min\{2,r\})$). It follows that $u_1(d)\geq 2$. Let $K$ be the subspace of $X$ consisting of the edges $ab,ac$ and $bc$.

Since $u_1(a)+u_2(a)=u_1(b)+u_2(b)=u_1(c)+u_2(c)=1$, at least three of the values $u_1(a),u_2(a),$ $u_1(b),u_2(b),$ $u_1(c),u_2(c)$ are $\geq\frac12$. Two of these three values necessarily correspond to the same function $u_i$, $i=1,2$. Therefore, without loss of generality (renaming the vertices and functions if necessary), we may assume that $u_1(b)\geq\frac12$ and $u_1(c)\geq\frac12$. Observe that this implies that $u_1(x)\geq\frac12$ for any point $x$ on the edge $bc$ (otherwise, we would have $u_2(b)\leq\frac12, u_2(c)\leq\frac12, u_2(x)=r>\frac12$ and $u_2(e)\geq 2$, contradicting unimodality as points $x$ and $e$ would lie in different components of $Q_2(r)$). But then, $u_1$ has at least one zero $z$ somewhere in the interior of the union of segments $bd$ and $cd$. (Otherwise $Q_1(r)$ would contain the whole cycle $bc,bd,cd$ for some $r>0$, contradicting unimodality.) This implies that $u_2(z)=r>1$, contradicting unimodality, as this means that $z$ and $e$ lie in different components of $Q_2(\min\{r,2\})$. This concludes the proof.
\end{proof}

Therefore, the monotonicity conjecture fails for $0<p_1<p_2=\infty$ as well.

\begin{remark}
Note that the basic idea underlying the proof is the fact that a cycle of odd length has chromatic number $3$. It remains an open question what exactly the connection between chromatic numbers and $\ucat$ is for general graphs.
\end{remark}

\subsection{Euclidean Plane: First Counterexample}\label{plane1}

In our construction of the counterexamples to the monotonicity conjecture on the two graphs above, we have exploited the fact that these graphs contain cycles. As we have also seen, the unimodal $p$-category is indeed monotone in $p$ for $X=\RR$. The question then arises whether it is essential that the space $X$ has non-trivial homology for such counterexamples to exist. We show that the answer to this question is also negative by constructing two counterexamples to the monotonicity conjecture in the Euclidean plane $X=\RR^2$.

The first counterexample we give is motivated by the first counterexample in the case when $X$ is a graph, so it bears some resemblance to it. For simplicity, whenever $p=(x,y)\in\RR^2$, we write $p^*=(y,x)$. We also adopt the convention that $\times$ binds more strongly than $\cup$. Let $a=(3,1)$ and define the following compact subset of $\RR^2$:
\[
K=[-1,1]\times[-1,1]\cup[-3,3]\times\{1\}\cup[1,3]\times\{-1\}\cup\{-3\}\times[1,3]\cup\{3\}\times[-3,-1].
\]

\hbox{}

\begin{center}
\includegraphics[width=170pt]{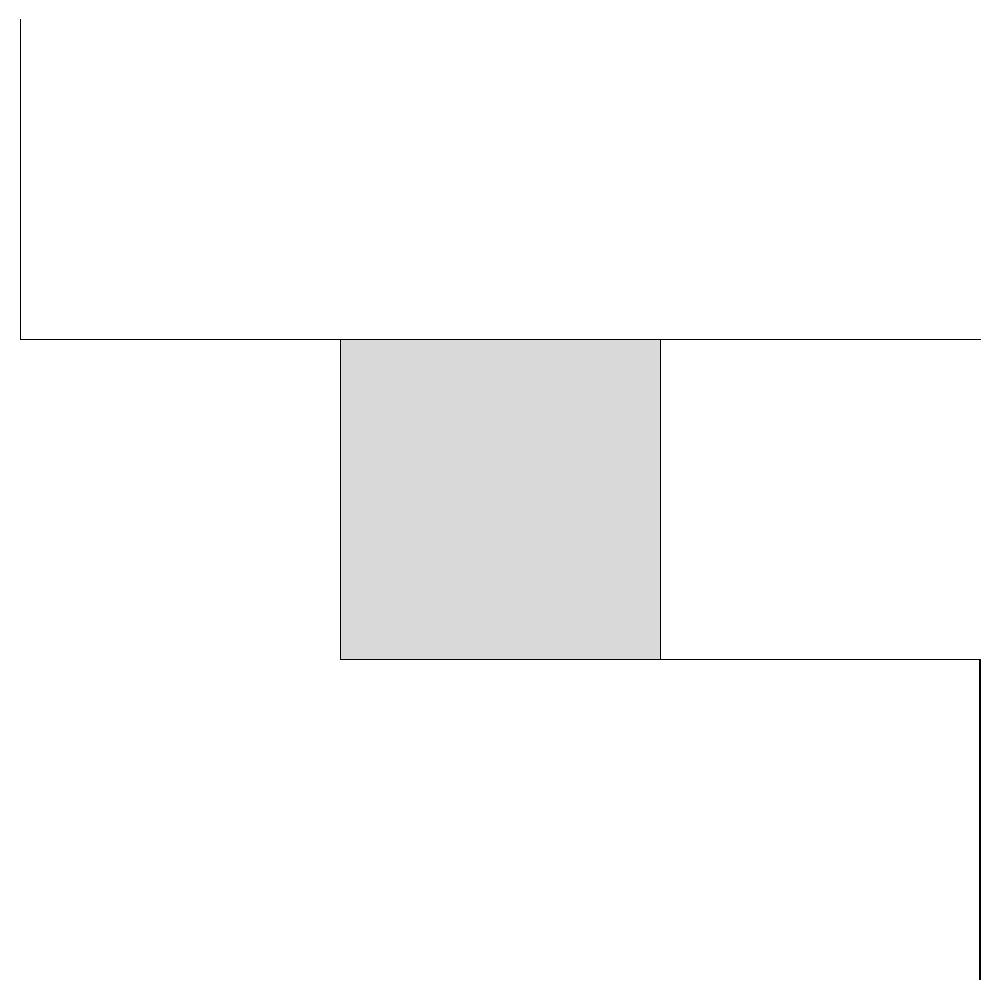}
\end{center}

\hbox{}

Our first counterexample can be concisely described using the $\infty$-distance to the set $K$. A more direct description is given in Appendix \ref{direct1}. We define the following functions $u_1,u_2,F,f:\RR^2\to[0,\infty)$:
\begin{align*}
u_1(p)&=\max\big\{0,1-d_\infty\big(p,K\big),5-5d_\infty\big(p,a\big)\big\},\\
u_2(p)&=u_1(p^*),\\
F(p)&=u_1(p)+u_2(p),\\
f(p)&=\sqrt{u_1(p)+u_2(p)}.
\end{align*}

For convenience, the following are the graphs of $u_1,u_2,F$ and $f$:

\hbox{}

\begin{center}
\includegraphics[width=170pt]{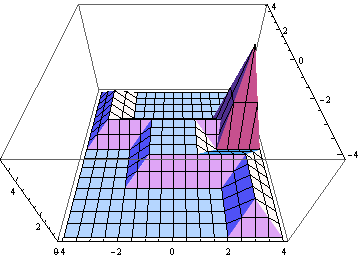}\qquad\includegraphics[width=170pt]{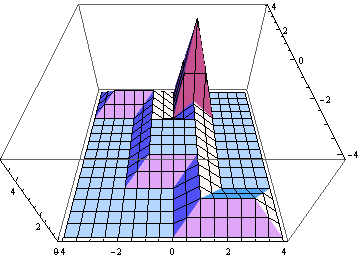}
\end{center}

\begin{center}
\includegraphics[width=170pt]{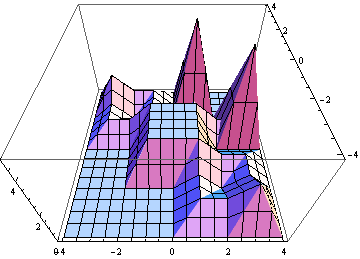}\qquad\includegraphics[width=170pt]{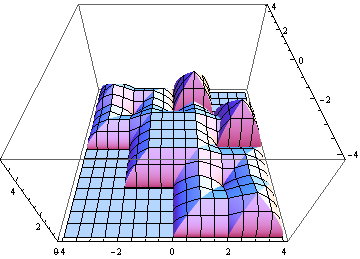}
\end{center}

\hbox{}

Note that each of these is continuous and compactly supported. Our main claim is the following:

\begin{proposition}\label{glavna}
The function $f$ is a counterexample to the conjecture, namely, $\ucat^2(f)=\ucat(F)=2$ and $\ucat(f)\geq3$.
\end{proposition}

The rest of this section is devoted to proving this proposition. First, we make a straightforward observation regarding the nature of the functions we have defined. (More details are given in Appendix \ref{direct1}.)

\begin{observation}\label{firstpart}
The functions $u_1,u_2:\RR^2\to[0,\infty)$ are piecewise linear and unimodal. The function $F:\RR^2\to[0,\infty)$ is piecewise linear. It is not unimodal, but is by definition the sum of two unimodal functions.
\end{observation}

In the proofs we will make use of the following points (see the picture below):
\[
\begin{aligned}
a&=(3,1)\\
a_1&=a+(-\tfrac45,-\tfrac45)\\
a_2&=a+(\tfrac45,-\tfrac45)\\
a_3&=a+(\tfrac45,\tfrac45)\\
a_4&=a+(-\tfrac45,\tfrac45)\\
a_5&=a+(-\tfrac45,0)\\
b&=a^*\\
b_1&=a_1^*\\
b_2&=a_2^*\\
b_3&=a_3^*\\
b_4&=a_4^*\\
b_5&=a_5^*
\end{aligned}\qquad\begin{aligned}
c_1&=(2,1)\\
c_2&=(\tfrac32,\tfrac12)\\
c_3&=(\tfrac32,-\tfrac12)\\
c_4&=(2,-1)\\
c_5&=(1,-2)\\
c_6&=(\tfrac12,-\tfrac32)\\
c_7&=(-\tfrac32,-\tfrac32)\\
c_8&=c_6^*\\
c_9&=c_5^*\\
c_{10}&=c_4^*\\
c_{11}&=c_3^*\\
c_{12}&=c_2^*\\
c_{13}&=c_1^*
\end{aligned}\qquad\begin{aligned}
d&=(3,-3)\\
d_1&=(2,-3)\\
d_2&=(\tfrac52,-\tfrac72)\\
d_3&=(\tfrac72,-\tfrac72)\\
d_4&=(\tfrac72,-\tfrac52)\\
d_5&=(3,-2)\\
e&=(-3,3)\\
e_1&=d_1^*\\
e_2&=d_2^*\\
e_3&=d_3^*\\
e_4&=d_4^*\\
e_5&=d_5^*
\end{aligned}\qquad\begin{aligned}
z_0&=(2,-2)\\
z_1&=(3,-1)\\
z_2&=(1,-3)\\
w_0&=z_0^*\\
w_1&=z_1^*\\
w_2&=z_2^*\\
q_1&=(-1,-1)\\
q_2&=(1,-1)\\
q_3&=(1,1)\\
q_4&=(-1,1)
\end{aligned}
\]
Finally, we also need the following sets:
\[
\begin{aligned}
A&=\overline{a_1a_2a_3a_4},\\
B&=\overline{b_1b_2b_3b_4},\\
C&=\overline{c_1c_2c_3c_4c_5c_6c_7c_8c_9c_{10}c_{11}c_{12}c_{13}},\\
D&=\overline{d_1d_2d_3d_4d_5},\\
E&=\overline{e_1e_2e_3e_4e_5},\\
K&=aw_2e\cup q_1z_1d\cup\overline{q_1q_2q_3q_4},
\end{aligned}\qquad\begin{aligned}
P_a&=a_5c_1,\\
P_b&=b_5c_{13},\\
Z_1&=c_5z_1d_1,\\
Z_2&=c_4z_2d_5,\\
W_1&=c_9w_1e_1,\\
W_2&=c_{10}w_2e_5.
\end{aligned}
\]

The following two observations are straightforward, so we omit their proofs.

\begin{observation}
The superlevel set $Q=f^{-1}[1,\infty)=F^{-1}[1,\infty)$ can be expressed as follows:
\[
Q = A\cup B\cup C\cup D\cup E\cup P_a\cup P_b\cup Z_1\cup Z_2\cup W_1\cup W_2.
\]
\end{observation}

\hbox{}

\begin{center}
\includegraphics[width=170pt]{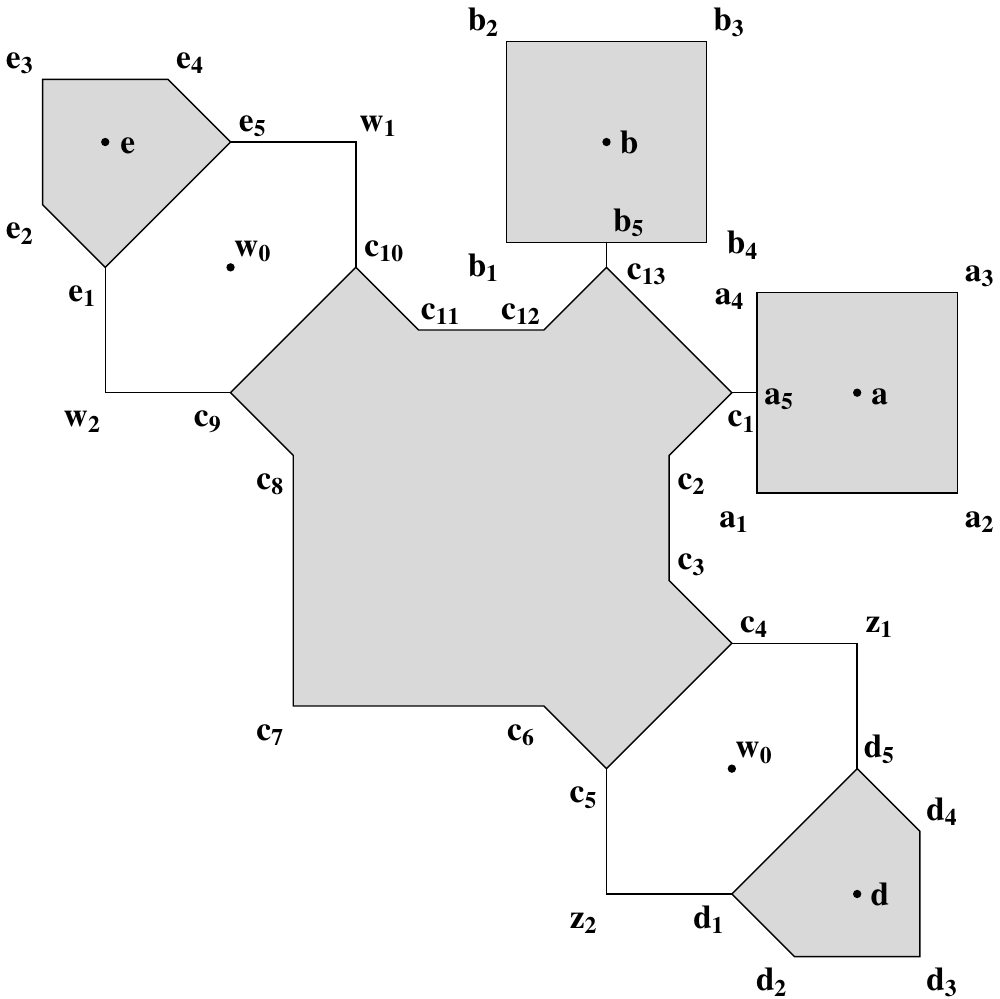}
\end{center}

\hbox{}

\begin{observation}\label{values}
The function $f$ has the following properties:
\begin{itemize}
\item $f(z_0)=f(w_0)=0$,
\item $f(a)=f(b)=\sqrt 5$,
\item $f(p)=1$ for all $p\in\partial Q$,\footnote{Here $\partial Q$ means the topological boundary of $Q$.}
\item $f(p)\leq\sqrt 2$ for all $p\notin A\cup B$.
\end{itemize}
\end{observation}

Before proceeding to the main proof, we require the following lemma.

\begin{lemma}\label{lemma}
Let $I=[0,1]$. Suppose $\gamma_1,\gamma_2:I\to I\times I$ are paths such that $\gamma_1(0)=(0,0)$, $\gamma_1(1)=(1,1)$, $\gamma_2(0)=(0,1)$ and $\gamma_2(1)=(1,0)$. Then $\gamma_1(I)\cap\gamma_2(I)\neq\emptyset$.
\end{lemma}

\begin{proof}
This follows directly from \cite[Lemma 2]{maehara} by taking $a=c=0$, $b=d=1$, $h=\gamma_1$ and $v=\gamma_2$.
\end{proof}

We can now prove our proposition.

\begin{proof}[Proof of Proposition \ref{glavna}]
The first part of the claim, that $\ucat^2(f)=2$, follows directly from Observation \ref{firstpart}.

To prove the second part, suppose for the sake of contradiction that there exists a decomposition $f=v_a+v_b$, where $v_a$ and $v_b$ are unimodal.\footnote{The notation is not meant to imply any relation with $a$ and $b$ at this point.} If $r>0$, define superlevel sets
\[
Q_a(r)=v_a^{-1}[r,\infty),\qquad\text{and}\qquad Q_b(r)=v_b^{-1}[r,\infty).
\]

We are especially interested in
\[
Q_a=Q_a(1)\qquad\text{and}\qquad Q_b=Q_b(1).
\]
Since $v_a$ and $v_b$ are unimodal, the sets $Q_a(r)$ and $Q_b(r)$ (where $r>0$) are all contractible. In particular, they are path-connected. If $r_1\leq r_2$, the inclusions
\[
Q_a(r_1)\supseteq Q_a(r_2),\qquad\text{and}\qquad Q_b(r_1)\supseteq Q_b(r_2)
\]
hold. Since $v_a,v_b\leq f$, we also have
\[
Q_a,Q_b\subseteq Q.
\]

By Observation \ref{values}, $f(p)\leq\sqrt2$ for $p\notin A\cup B$. If $v_a(p)\geq 1$ and $v_b(p)\geq 1$, we have $f(p)\geq 2$, so $p\in A\cup B$. This shows that $Q_a\cap Q_b\subseteq A\cup B$.

Observe that $v_a(a)+v_b(a)=\sqrt5$ holds. By interchanging $v_a$ and $v_b$, if necessary, we may assume without loss of generality that $v_a(a)\geq\frac{\sqrt5}2$. (Which is in fact the reason why we chose this notation for $v_a$ and $v_b$.)

\begin{step}
The following inequalities are satisfied:
\begin{itemize}
\item $v_a(p)\leq 1$ for all $p\in Q\setminus A$,
\item $v_b(p)\leq 1$ for all $p\in Q\setminus B$,
\item $v_a(a)\geq\sqrt5-1$ and
\item $v_b(b)\geq\sqrt5-1$.
\end{itemize}
\end{step}

\begin{proof}
To prove the first of these, it suffices to show that $Q_a(r)\subseteq A$ for all $r>1$. We know that $Q_a(r)\subseteq Q_a\subseteq Q$ holds. Note that $Q\setminus\{a_5\}$ is not path connected. The path component of $a$ in $Q\setminus\{a_5\}$ is a subset of $A$. Since $v_a(a_5)\leq f(a_5)=1$, the point $a_5$ is not contained in $Q_a(r)$. Therefore, the path component of $a$ in $Q_a(r)$ is also a subset of $A$ and we are done.

Since $v_a(b)+v_b(b)=\sqrt5$, this means that $v_b(b)\geq\sqrt5-1$. The fact that $v_b(p)\leq 1$ holds for all $p\in Q\setminus B$ is obtained by a symmetric argument. This also implies that $v(a)\geq\sqrt5-1$ holds.
\end{proof}

\begin{step}\label{points}
Let $S=\{c_9,c_{10}\}$ or $S=\{c_4,c_5\}$. Then there exist points $p,q\in S$ such that $v_a(p)=1$ and $v_b(q)=1$.
\end{step}

\begin{proof}
We show this for $S=\{c_4,c_5\}$, the other proof is symmetric. Note that $f(z_0)=0$, so $v_a(z_0)=v_b(z_0)=0$. For each $t\in(0,1)$ define $c_4(t)=(1-t)c_4+tc_3$ and $c_5(t)=(1-t)c_5+tc_6$. Define
\[
L(t) = D\cup Z_1\cup Z_2\cup c_4c_4(t)\cup c_4(t)c_5(t)\cup c_5(t)c_5.
\]
Note that $L(t)$ is compact, so $v_b$ attains its minimum in $L(t)$, say $v_b(p_t)=m_t$. But $m_t$ cannot be positive: if $m_t>0$, we would have $L(t)\subseteq Q_b(m_t)$ and $z_0\notin Q_b(m_t)$, implying the existence of a retraction $Q_b(m_t)\to L(t)$, which is impossible, since $Q_b(m_t)$ is contractible.

Therefore, $v_b(p_t)=0$, or since $f$ is at least $1$ on $L(t)$, $v_a(p_t)\geq 1$. But we know that $v_a(p)\leq 1$ holds for $p\notin A$, so we must have $v_a(p_t)=1$. This also implies that $p\in\partial Q$.

There are now two possibilities: if for some $t$ we obtain $p_t\in D\cup Z_1\cup Z_2$, we also have a path in $Q_a$ from $p_t$ to $a$, which must necessarily pass either through $c_4$ or $c_5$, since $p_t$ and $a$ lie in different path components of $Q\setminus\{c_4,c_5\}$. The only remaining possibility is that $p_t\in c_4c_4(t)\cup c_4(t)c_5(t)\cup c_5(t)c_5$ for all $t\in(0,1)$. Since $p_t\in\partial Q$, this means that $p_t\in c_4c_4(t)\cup c_5c_5(t)$ holds for all $t$. Therefore we may choose a convergent subsequence $(p_{t_n})_n$ of $(p_{\frac1k})_k$ such that $v_a(p_{t_n})=1$ for all $n$. This sequence converges either to $c_4$ or $c_5$, so one of $v_a(c_4)=1$, $v_a(c_5)=1$ must hold.

The proof that one of $v_b(c_4)=1$, $v_b(c_5)=1$ holds is symmetric.
\end{proof}

\begin{step}
The equalities $v_a(c_1)=1$ and $v_b(c_{13})=1$ hold.
\end{step}

\begin{proof}
This follows from the previous step. Let $S$ be any of the two sets from the previous step. Let $p,q\in S$ be points such that $v_a(p)=1$ and $v_b(q)=1$. There are paths in $Q_a$ and $Q_b$ from $a$ to $p$ and from $b$ to $q$, respectively. The first path must cross $c_1$, since $a$ and $p$ lie in different path components of $Q\setminus\{c_1\}$ and the second one must cross $c_{13}$, since $b$ and $q$ lie in different path components of $Q\setminus\{c_{13}\}$.
\end{proof}

\begin{step}
The results of the previous steps contradict each other.
\end{step}

\begin{proof}
Let $p\in\{c_9,c_{10}\}$ have the property that $v_a(p)=1$ and let $q$ be the unique element of $\{c_9,c_{10}\}\setminus\{p\}$. Let $q'\in\{c_4,c_5\}$ have the property that $v_b(q')=1$ and let $p'\in\{c_4,c_5\}\setminus\{q'\}$. These points exist by step \ref{points}. There is a path $\gamma_a:I\to Q_a$ from $p$ to $c_1$ and a path $\gamma_b:I\to Q_b$ from $q'$ to $c_{13}$.

Observe that $\gamma_a(I)\cap(B\cup P_b)=\emptyset$ and $\gamma_b(I)\cap(A\cup P_a)=\emptyset$, since $c_1\notin Q_b$ and $c_{13}\notin Q_a$. We can also ensure that $\gamma_a(I)\cap(E\cup W_1\cup W_2)=\{p\}$ and $\gamma_a(I)\cap(A\cup P_a)=\{c_1\}$ by cropping the path at both ends (formally, take $t_1=\sup\{t\in[0,1]\mid\gamma_a(t)=p\}$ and $t_2=\inf\{t\in[t_1,1]\mid\gamma_a(t)=c_1\}$ and reparametrize the restriction of $\gamma_a$ to $[t_1,t_2]$). In the same way we can ensure that $\gamma_b(I)\cap(D\cup Z_1\cup Z_2)=\{q'\}$ and $\gamma_b(I)\cap(B\cup P_b)=\{c_{13}\}$. Assume, therefore, without loss of generality that $\gamma_a$ and $\gamma_b$ have these properties.

Now, let $R_a:Q_a\to Q_a\setminus\big((D\cup Z_1\cup Z_2)\setminus\{p'\}\big)$ be the retraction that takes the points of $(D\cup Z_1\cup Z_2)\setminus\{p'\}$ to $p'$ and let $R_b:Q_b\to Q_b\setminus\big((E\cup W_1\cup W_2)\setminus\{q\}\big)$ be the retraction that takes the points of $(E\cup W_1\cup W_2)\setminus\{q\}$ to $q$. These retractions are well defined, since $q'\notin Q_a$ and $p\notin Q_b$. Note that $R_a\circ\gamma_a$ is a path in $Q_a\cap C$ and $R_b\circ\gamma_b$ is a path in $Q_b\cap C$ and these two paths have the same endpoints as $\gamma_a$ and $\gamma_b$, respectively.

Note that $\partial C$ is a Jordan curve and that $\partial C\setminus\{c_1,p\}$ has two path components, each of which contains exactly one of the points $c_{13},q'$. Using the Jordan-Schönflies theorem \cite[Chapter III]{bing}, we obtain the situation in which Lemma \ref{lemma} applies and we can conclude that the paths $R_a\circ\gamma_a$ and $R_b\circ\gamma_b$ intersect somewhere in $C$. But this is a contradiction, as we have already established that $Q_a\cap Q_b\subseteq A\cup B$.
\end{proof}

This contradiction shows that $\ucat(f)\geq 3$, which concludes the proof of our proposition.
\end{proof}

\begin{remark}
The homology of the space is trivial, but the first homology of the superlevel sets is not. This seems to be an essential feature of the counterexample, as it enables us to force certain values upon the functions in an unimodal decomposition. An explicit unimodal decomposition of length $3$ can be constructed for $f$, but we do not describe it here, as it has no bearing on the validity of the counterexample. Furthermore, by modifying the function values at the vertices, we can obtain such counterexamples for any pair $0<p_1<p_2<\infty$. We expect such counterexamples to exist on $\RR^m$ for any $m\geq2$.
\end{remark}

\subsection{Euclidean Plane: Second Counterexample}\label{plane2}

We need a different idea to show that the monotonicity conjecture fails for $X=\RR^2$ also in the case of $0<p_1<p_2=\infty$. The example we give is completely analogous to the second example in the case of graphs.

Let $d_0=(-4,0)$ and $e_0=(4,0)$. Define the following compact subset of $\RR^2$:
\[
K=[-4,4]\times\{0\}\cup\{-6,0\}\times[-6,6]\cup[-6,0]\times\{-6,6\}\cup\{-2,2\}\times[-4,4]\cup[-2,2]\times\{-4,4\}.
\]

This can be pictured as follows:

\hbox{}

\begin{center}
\includegraphics[width=170pt]{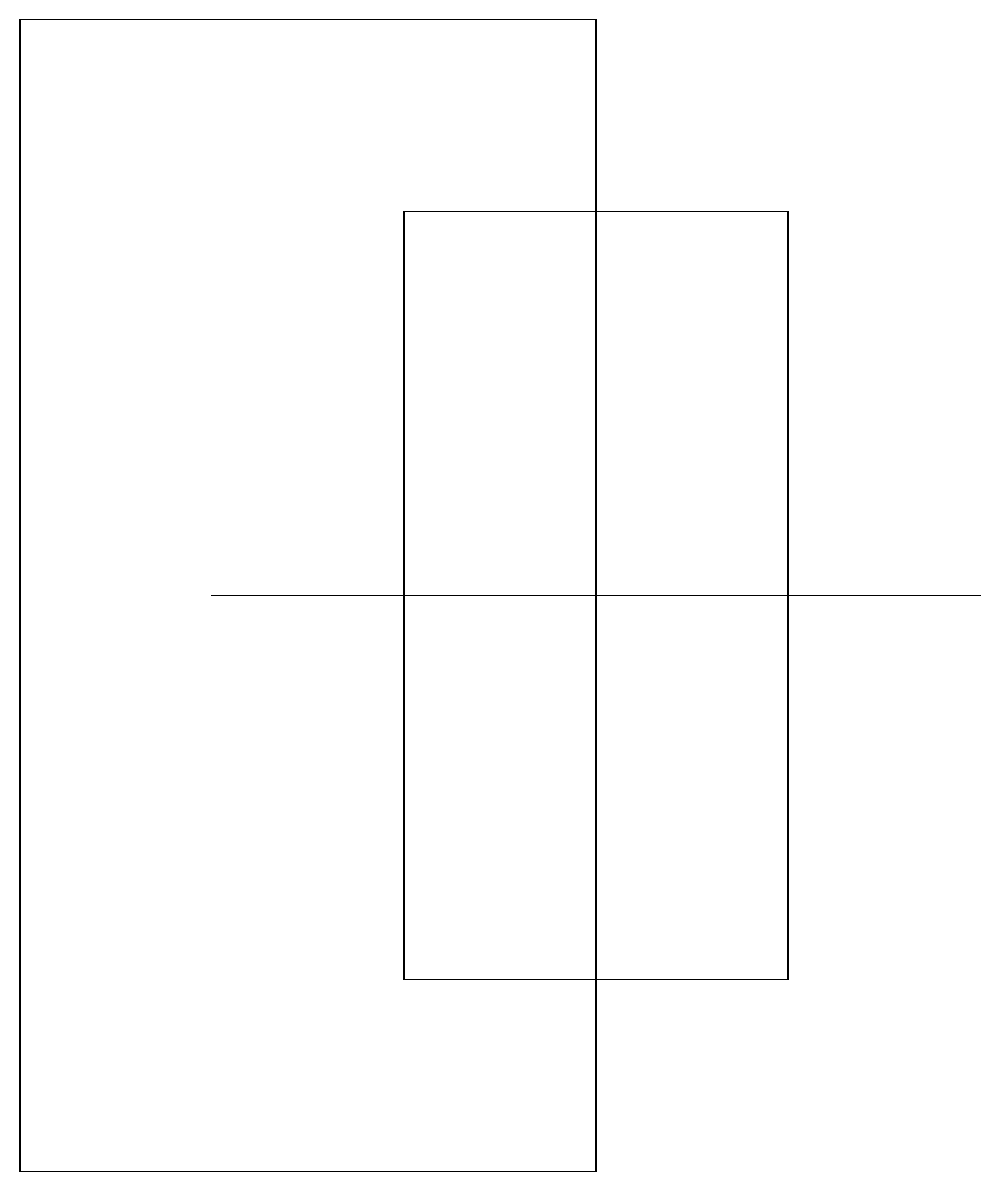}
\end{center}

\hbox{}

Define $f:\RR^2\to[0,\infty)$:
\begin{align*}
f(p)&=\max\big\{0,1-d_\infty\big(p,K\big),3-3d_\infty\big(p,d_0\big),3-3d_\infty\big(p,e_0\big)\big\}.
\end{align*}

Note that $f$ is continuous, compactly supported and piecewise linear. Further details of this latter fact are given in Appendix \ref{direct2}, where a more direct description of $f$ is given. Here is the graph of $f$:

\begin{center}
\includegraphics[width=170pt]{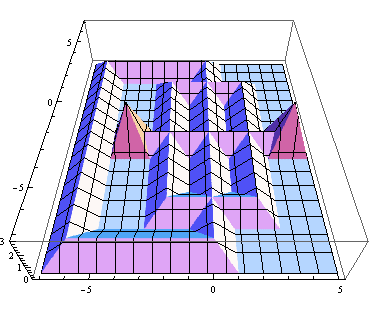}
\end{center}

Our main claim is the following:

\begin{proposition}\label{glavna_infty}
The function $f$ is a counterexample to the conjecture. Concretely, $\ucat^\infty(f)=2$ and $\ucat(f)=3$.
\end{proposition}

The rest of this section is devoted to proving this proposition. Computing $\ucat^{\infty}(f)$ is completely analogous to the second counterexample in the case of graphs. For this reason, we prefer to omit the proof of the following observation here (for a detailed construction, see Appendix \ref{direct2}). At this point, we simply note that $f$ is not unimodal and that it is straightforward to construct a unimodal $\infty$-decomposition of length $2$.

\begin{observation}\label{ucat2}
The unimodal $\infty$-category of $f$ is $\ucat^{\infty}(f)=2$.
\end{observation}

Proving that $\ucat(f)=3$ is also very similar as in the graph case. Most of the action takes place in the sublevel set $Q=f^{-1}[1,\infty)$, so before beginning the proof, we describe it explicitly. The notation we use is similar as in the first counterexample in $\RR^2$. In addition to $d_0=(-4,0)$ and $e_0=(4,0)$, we define the following points:
\[
\begin{aligned}
d_1&=(-\tfrac{14}3,-\tfrac23),\\
d_2&=(-\tfrac{10}3,-\tfrac23),\\
d_3&=(-\tfrac{10}3,\tfrac23),\\
d_4&=(-\tfrac{14}3,\tfrac23),\\
d_5&=(-2,-4),\\
d_6&=(-2,0),\\
d_7&=(-2,4),
\end{aligned}\qquad\begin{aligned}
e_1&=(\tfrac{14}3,-\tfrac23),\\
e_2&=(\tfrac{10}3,-\tfrac23),\\
e_3&=(\tfrac{10}3,\tfrac23),\\
e_4&=(\tfrac{14}3,\tfrac23),\\
e_5&=(2,-4),\\
e_6&=(2,0),\\
e_7&=(2,4),
\end{aligned}\qquad\begin{aligned}
a&=(0,4),\\
b&=(0,0),\\
c&=(0,-4),\\
k_1&=(0,-6),\\
k_2&=(-6,-6),\\
k_3&=(-6,6),\\
k_4&=(0,6).
\end{aligned}
\]

Next, we define the following sets:
\begin{align*}
D=\overline{d_1d_2d_3d_4},\quad E=\overline{e_1e_2e_3e_4},\quad P=d_5e_5e_7d_7d_5,\quad R=d_0e_0,\quad K=k_1k_2k_3k_4k_1.
\end{align*}
We can now give a simple description of $Q$.

\begin{observation}
The superlevel set $Q=f^{-1}[1,\infty)$ can be expressed as follows:
\[
Q = D\cup E\cup P\cup R\cup K.
\]
\end{observation}

\hbox{}

\begin{center}
\includegraphics[width=170pt]{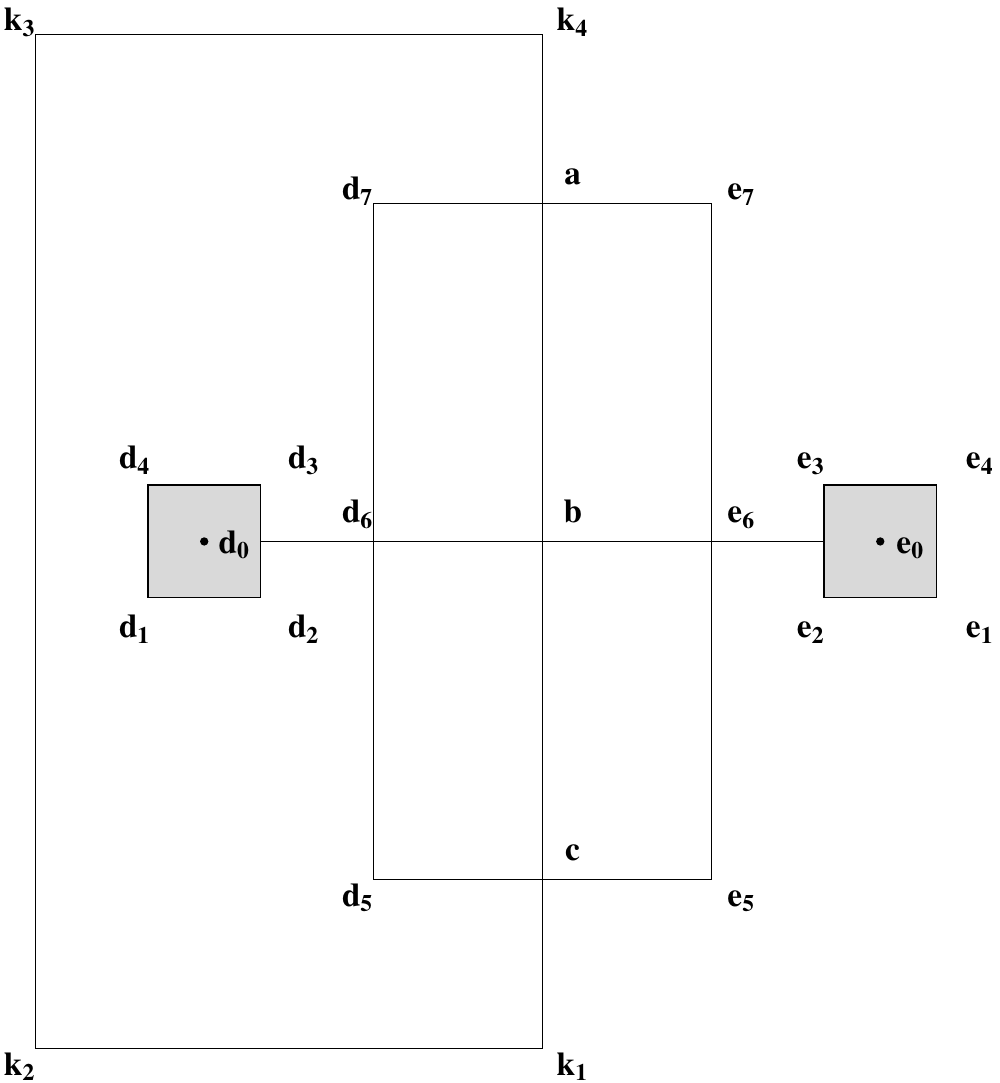}
\end{center}

\hbox{}

We are now ready to compute the unimodal category of $f$. The proof in the graph case relied on the fact that $a,b,c$ are local cut points of the graph and the various superlevel sets they appear in. In the case of $\RR^2$, which has no local cut points, a more careful case by case analysis is required.

\begin{proposition}\label{ucat3}
The unimodal category of $f$ is $\ucat(f)=3$.
\end{proposition}

\begin{proof}
A unimodal decomposition of length $3$ can easily be constructed explicitly and is described in Appendix \ref{direct2}. To show that there is no unimodal decomposition of length $2$, we argue by contradiction. Suppose $f=u_1+u_2$ is such a decomposition. If $r>0$, define superlevel sets
\[
Q_1(r)=u_1^{-1}[r,\infty)\qquad\text{and}\qquad Q_2(r)=u_2^{-1}[r,\infty).
\]
Without loss of generality, we may assume that $u_1(d_0)\geq u_2(d_0)$, in other words, $u_1(d_0)\geq\frac32$. By unimodality, it follows that $u_1(e_0)\leq 1$ (otherwise $u_1(e_0)=r>1$ and the points $d_0$ and $e_0$ would have to lie in separate components of $Q_1(\min\{\frac32,r\})$, separated by $K$). It follows that $u_2(e_0)\geq 2$. Using unimodality again, we have $u_2(d_0)\leq 1$ (otherwise $u_2(d_0)=r>1$ and the points $d_0$ and $e_0$ would have to lie in separate components of $Q_2(\min\{2,r\})$, separated by $K$). It follows that $u_1(d_0)\geq 2$.

Observe that for $i=1,2$, $u_i$ must have a zero somewhere in $K$ (if $u_i(x)\geq r>0$ for all $x\in K$, the inclusion $K\hookrightarrow Q_i(r)$ is non-trivial on $H_1$), say $u_i(w_i)=0$. It follows that $u_{i}(w_{3-i})=1$, so $Q_i(1)\cap K\neq\emptyset$. Since $Q_1(1)$ is contractible, there is a path in $Q\supseteq Q_1(1)$ from $d_0$ to $w_2$. Therefore, there is a point $q_1\in\{a,b,c\}$ such that $u_1(q_1)=1$ (there is no path from $d_0$ to $K\setminus\{a,b,c\}$ in $Q\setminus\{a,b,c\}$). Similarly, there is a point $q_2\in\{a,b,c\}$ such that $u_2(q_2)=1$. Since there are paths in $Q_1(1)$ from $d_6$ to $q_1$ and in $Q_2(1)$ from $e_6$ to $q_2$, we also have $u_1(d_6)=1$ and $u_2(e_6)=1$. Note that $q_1$ and $q_2$ are distinct and let $q_3$ be the third point in $\{a,b,c\}$.

There are now six cases, each of which leads to a contradiction. To avoid treating these cases separately, we proceed as follows. Let $Z_1$ be the unique subspace of $Q$ which is homeomorphic to the circle, contains $d_6$ and $q_1$, but does not contain $e_6$ and $q_2$. Similarly, let $Z_2$ be the unique subspace of $Q$ which is homeomorphic to the circle, contains $e_6$ and $q_2$, but does not contain $d_6$ and $q_1$. For instance, if $q_2=c$, we have $Z_1=abd_6d_7a$ and if $q_2=b$ we have $Z_1=ad_7d_6d_5ck_1k_2k_3k_4a$.

Now note that $u_1$ must have a zero $z_1$ somewhere in $Z_1$ (otherwise $Z_1\hookrightarrow Q_1(r)$ would be non-trivial on $H_1$ for some $r>0$). Therefore $u_2(z_1)=1$ and by unimodality there is a path in $Q_2(1)$ from $q_2$ to $z_1$. This path must contain the point $q_3$ (there is no path from $q_2$ to $Z_1\setminus\{q_1,q_3,d_6\}$ in $Q\setminus\{q_1,q_3,d_6\}$), therefore $u_2(q_3)=1$. Similarly, $u_2$ must have a zero $z_2$ somewhere in $Z_2$. Therefore $u_1(z_2)=1$ and by unimodality, there is a path in $Q_1(1)$ from $q_1$ to $z_2$ which must contain the point $q_3$. We conclude that $u_1(q_3)=1$. As $u_1(q_3)+u_2(q_3)=2\neq 1=f(q_3)$, we have reached a contradiction, thus concluding the proof.
\end{proof}

\begin{remark}
Note that this counterexample can be modified to work in any $\RR^m$, $m\geq2$.
\end{remark}

\subsection{Proof in $\RR^2$ if the Morse-Smale Graph is a Tree}\label{plane}

Hickok, Villatoro and Wang describe in \cite{hickok} how to compute the unimodal category of a nonresonant function $f:\RR^2\to[0,\infty)$ whose Morse-Smale graph is a tree. A {\em nonresonant} function is a Morse function all of whose critical values are distinct \cite{nicolaescu}. We show that their results in fact also imply that the monotonicity conjecture is true for such functions, which appears to have gone unnoticed, even though it follows from their result almost immediately. So at least in the case of Morse functions, the presence of cycles is an essential feature of counterexamples to the monotonicity conjecture.

\begin{definition}[\cite{hickok}]
A {\em Morse-Smale graph} associated to a Morse function $f:\RR^2\to[0,\infty)$ is a weighted graph, embedded in $\RR^2$, whose vertices are the local maxima of $f$ and whose edges are associated to the saddles of $f$ in the following way: corresponding to each saddle, the graph has precisely one edge, which is realized as a path connecting two local maxima and passing through the saddle, so that the function values are decreasing on the portion of the path between each maximum and the saddle. The weight corresponding to a maximum $m\in\RR^2$ is given by $f(m)$ and the weight corresponding to the edge associated to the saddle $s\in\RR^2$ is given by $f(s)$.
\end{definition}

Note that a Morse-Smale graph is not uniquely determined by the function.

\subsubsection{Path values}

Following \cite{hickok}, we are going to use the concept of {\em the path value}, however, we phrase it in a slightly different way, which we feel should be more amenable to generalization. In \cite{hickok}, this concept is defined using the concept of the Morse-Smale graph. We prefer to bypass this using a somewhat more general definition, that applies to general topological spaces. The benefit of this approach is that we obtain new lower bounds for general topological spaces.

\begin{definition}
Let $f:X\to[0,\infty)$ and $x_1,x\in X$. Then the {\em path value} from $x_1$ to $x$ is the number
\[
\pv(x_1,x)=\sup_{\gamma\in\Gamma(x_1,x)}\min_{t\in[0,1]} f(\gamma(t)),
\]
where $\Gamma(x_1,x)$ is the set of all paths $\gamma:(I,0,1)\to(X,x_1,x)$.
\end{definition}

The concept of path value can be used to obtain lower bounds for $\ucat$:

\begin{proposition}\label{lower_bound}
Suppose $f:X\to[0,\infty)$ is such that $\ucat(f)\leq n\in\NN$. Then there exist points $x_1,x_2,\ldots,x_n$ such that
\[
\sum_{i=1}^n\pv(x_i,x)\geq f(x)
\]
holds for each $x\in X$.
\end{proposition}

\begin{proof}
Let $f=\sum_{i=1}^n u_i$ be a unimodal decomposition and choose points $x_1,\ldots,x_n$ so that for each $i$, $x_i$ is a maximum of $u_i$. Now, observe that for each $x\in X$ we have $u_i(x)\leq\pv(x_i,x)$. This is because $u_i^{-1}[u_i(x),\infty)$ is path connected, so there exists a path $\gamma$ from $x_i$ to $x$ such that $u_i(\gamma(t))\geq u_i(x)$ holds for all $t$. Therefore
\[
u_i(x)=\min_{t\in[0,1]}u_i(\gamma(t))\leq\min_{t\in[0,1]}f(\gamma(t))\leq\pv(x_i,x).
\]
This implies
\[
f(x)=\sum_{i=1}^n u_i(x)\leq\sum_{i=1}^n \pv(x_i,x).
\]
\end{proof}

The converse of this proposition is not generally true, however, as the authors of \cite{hickok} observe, it is almost true in the case $X=\RR^2$.

\begin{theorem}[\cite{hickok}, Proposition 4.3]\label{decomposition}
Suppose $f:\RR^2\to[0,\infty)$ is a nonresonant function whose Morse-Smale graph is a tree and there are local maxima $x_1,\ldots,x_n$ (not necessarily distinct) such that
\[
\sum_{i=1}^n\pv(x_i,x)>f(x)
\]
holds for each local maximum $x\neq x_i$ ($i=1,2,\ldots,n$). Then $\ucat(f)\leq n$.
\end{theorem}

This result relies on the fact that a nonresonant function whose Morse-Smale graph is a tree always has a Morse-Smale graph which is a path. This allows the authors to describe a general function of this kind in terms simple enough to allow for the construction of an explicit unimodal decomposition, which is what they proceed to do.

\begin{remark}
We note that the assumption of nonresonance which seems to have been overlooked by the authors of \cite{hickok} is crucial here, otherwise it could happen that the Morse-Smale graph cannot be converted into a path. For instance, a function whose critical sublevel sets are as depicted in the following picture, has a Morse-Smale graph which is a tree, but which cannot be converted into a path by the procedure described in \cite{hickok}.
\begin{center}
\includegraphics[width=170pt]{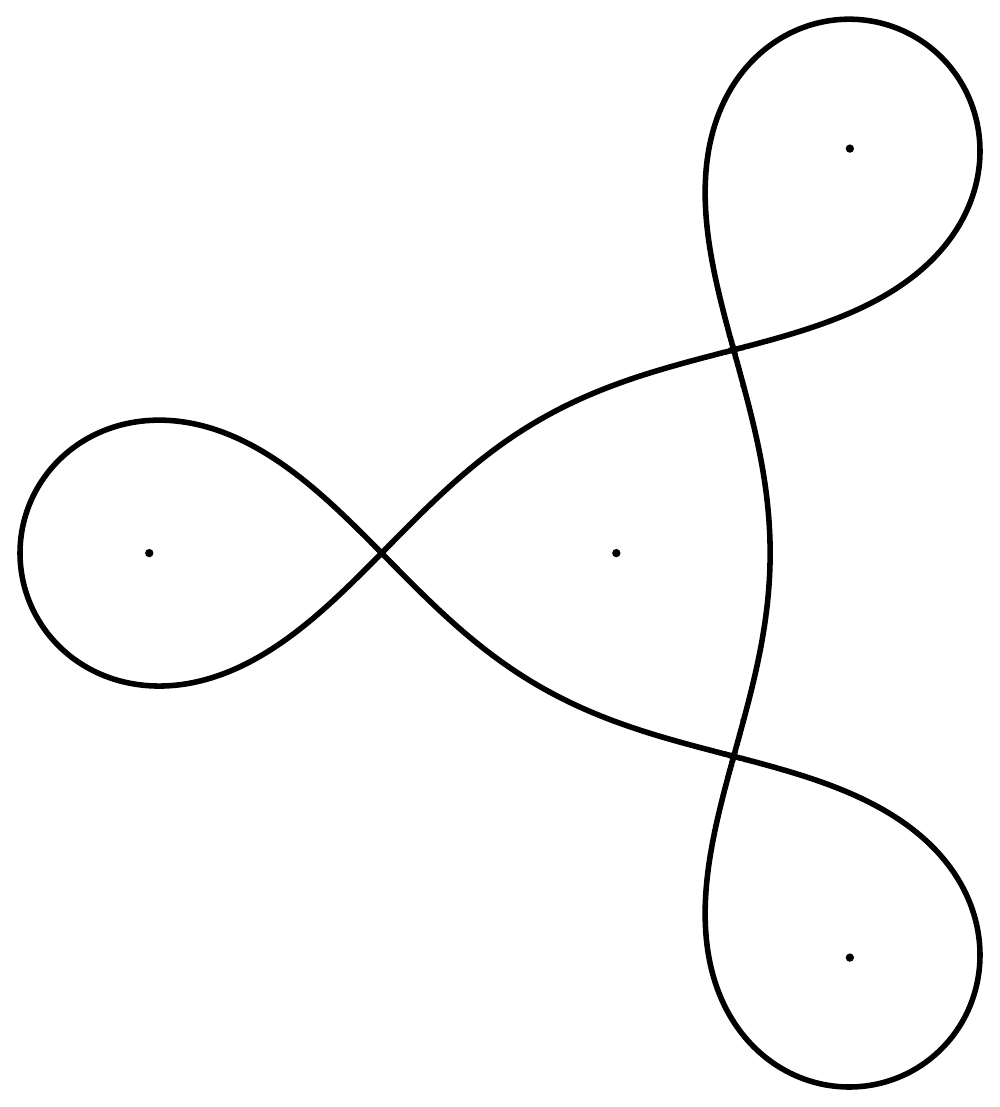}
\end{center}
\end{remark}

\subsubsection{Monotonicity}

Using the results of \cite{hickok}, we can now prove that the monotonicity conjecture holds for any nonresonant function $f:\RR^2\to[0,\infty)$ whose Morse-Smale graph is a tree. This follows almost immediately from the characterization using path values.

\begin{theorem}
Suppose $f:\RR^2\to[0,\infty)$ is a nonresonant function whose Morse-Smale graph is a tree and $0<p_1<p_2\leq\infty$. Then
\[
\ucat^{p_1}(f)\leq\ucat^{p_2}(f).
\]
\end{theorem}

\begin{proof}
Let $g=f^{p_1}$ and $p=\frac{p_2}{p_1}$. Note that $g$ is again nonresonant and its Morse-Smale graph is still a tree. By Lemma \ref{powers}, the statement we wish to prove is equivalent to
\[
\ucat(g)\leq\ucat(g^p).
\]
Suppose $\ucat(g)\leq n$. Then by Theorem \ref{lower_bound}, there exist points $x_1,\ldots,x_n$ such that
\[
\sum_{i=1}^n\pv(x_i,x)\geq g(x).
\]
Now, if $x$ is a local maximum of $g$ distinct from all $x_i$, $i=1,2,\ldots,n$, at least two path values $\pv(x_i,x)$ must be nonzero, since otherwise they cannot sum to $\geq g(x)$. By the usual norm inequalities, this immediately implies
\[
\sum_{i=1}^n\pv(x_i,x)^p>g(x)^p,
\]
which, using Theorem \ref{decomposition}, yields
\[
\ucat(g^p)\leq n,
\]
as desired.
\end{proof}

\subsection{Discussion}

\subsubsection{Multimodal Functions}

A Morse function on a manifold $M$ contains a lot of information about the topology of a manifold. For instance, the following fundamental theorem tells us that a manifold can be reconstructed up to homotopy by the process of attaching handles. In particular, this allows us to describe the homology of the manifold, using the observation that attaching an $i$-handle must either kill a homology class in $H_{i-1}$ or create a homology class in $H_i$.

\begin{theorem}[\cite{milnor}]\label{morsetheory}
Suppose $f:M\to\RR$ is a Morse function, $p\in M$ is a critical point of index $i$ and $f(p)=a$ the corresponding critical value. For every $x\in\RR$ let $M_x=f^{-1}(-\infty,x]$. Then if $[a-\epsilon,a+\epsilon]$ contains no other critical values of $f$, the sublevel set $M_{a+\epsilon}$ is obtained (up to homotopy) from $M_{a-\epsilon}$ by attaching an $i$-handle.
\end{theorem}

Given the successful application of the concept of Morse-Smale graphs which are trees in the case of $\RR^2$, it would be desirable to have a similar concept in $\RR^m$ for $m>2$. In fact, such a graph can be defined if the function $f:\RR^m\to[0,\infty)$ only has critical points of indices $m$ and $m-1$.

However, we may be able to generalize this a bit. The main problem which we are trying to avoid using the requirement that the Morse-Smale graph is a tree, is the presence of cycles in the superlevel sets. As a more general notion, akin to unimodality, that captures this, we propose the following:

\begin{definition}
A function $f:X\to[0,\infty)$ is {\em multimodal} if there is a $M>0$ such that each superlevel set $f^{-1}[c,\infty)$ is homotopy equivalent to a finite set of points for $0<c\leq M$ and empty for $c>M$.
\end{definition}

Such a function cannot have any cycles that would allow us to force certain values upon the unimodal summand in the decomposition as we did with the counterexamples in the plane, so it seems more likely that the following question could admit a positive answer:

\begin{question}
Suppose $X$ is a sufficiently nice space (for instance a manifold) and $f:X\to[0,\infty)$ is a multimodal function. Does this imply that $\ucat^p(f)$ is monotone in $p$?
\end{question}

To demonstrate that this is indeed a generalization of the case studied in \cite{hickok}, we prove the following result.

\begin{proposition}
Suppose $f:\RR^m\to[0,\infty)$ is a multimodal nonresonant function with compact support. Then $f$, restricted to $f^{-1}(0,\infty)$, has only critical points of index $m$ and $m-1$.
\end{proposition}

\begin{proof}
Consider instead the function $-f$ and let $M_x=(-f)^{-1}(-\infty,x]$ for each $x<0$. By Morse theory, it suffices to prove that this function only has critical points of index $0$ and $1$. Suppose $-f$ has a critical point $p$ of index $i>1$ and choose it so that the corresponding critical value $a=-f(p)$ is minimal. Suppose $[a-\epsilon,a+\epsilon]$ contains no other critical values. By Theorem \ref{morsetheory}, up to homotopy, passing a critical point of index $i$ corresponds to attaching an $i$-handle, so $M_{a+\epsilon}$ is homotopy equivalent to $M_{a-\epsilon}$ with an $i$-handle attached. However, attaching an $i$-handle must either kill a homology class in $H_{i-1}$ or create a homology class in $H_i$. In both cases, this is a contradiction: by multimodality, $M_{a-\epsilon}$ and $M_{a+\epsilon}$ are both finite unions of contractible sets, so such homology classes cannot exist.
\end{proof}

\subsubsection{$\pi_0$-Unimodal Functions}

Cycles seem to be fundamental to the nature of our counterexamples to monotonicity. It would be interesting to have a coarser version of unimodality that ignores cycles. If in the condition of unimodality, we use path-connectedness instead of contractibility, we obtain the following definition:

\begin{definition}
A continuous function $u:X\to[0,\infty)$ is {\em $\pi_0$-unimodal} if there is a $M>0$ such that the superlevel sets $u^{-1}[c,\infty)$ are path-connected for $0<c\leq M$ and empty for $c>M$.
\end{definition}

This yields the following notion:

\begin{definition}
Let $p\in(0,\infty)$. The {\em $\pi_0$-unimodal $p$-category} $\ucat_{\pi_0}^p(f)$ of a function $f:X\to[0,\infty)$ is the minimum number $n$ of $\pi_0$-unimodal functions $u_1,\ldots,u_n:X\to[0,\infty)$ such that pointwise, $f=(\sum_{i=1}^n u_i^p)^{\frac1p}$. The {\em $\pi_0$-unimodal $\infty$-category} is defined analogously, using the $\infty$-norm instead.
\end{definition}

Other variations of this concept are also possible.

\begin{question}
Is $\ucat_{\pi_0}^p(f)$ easier to compute than $\ucat^p(f)$? Is it monotone in $p$?
\end{question}

\subsubsection{Cohomological Approach}

In the study of Lusternik-Schnirelmann category \cite{cornea}, cohomological methods have been very successful. For instance, one of the basic bounds that is used is
\[
\operatorname{cup}_R(X)<\cat(X).
\]
A natural question is whether such cohomological methods can be developed for the study of unimodal category. First an appropriate cohomology theory is needed. If we expect such a theory to be functorial, it must be defined on a suitable category of functions. One natural candidate seems to be the category whose objects are functions $X\to[0,\infty)$ and a morphism between two such functions $X_1\to[0,\infty)$ and $X_2\to[0,\infty)$ is an appropriate commutative triangle. To be useful, one property such a cohomology theory should have is that the cohomology ring of a unimodal function is trivial. Furthermore, it should be additive with respect to functions with disjoint supports. The other properties are not immediately clear. One issue that arises is that the notion of unimodal category does not seem to be homotopy invariant in any sense. This is analogous to the notion of $\gcat$, which is also not a homotopy invariant.

\begin{question}
Is such a cohomological approach possible?
\end{question}

\section{Conclusions}

The concept of unimodal category, as envisioned by Baryshnikov and Ghrist in \cite{baryshnikov}, has its origins in statistics. However, the definition of the concept requires much less than this, namely a nonnegative function $f:X\to[0,\infty)$. We feel that the question of how many functions with a unique local maximum are needed to express such an $f$ is very natural from the point of view of mathematical analysis. For this reason, we find it somewhat surprising that this question has received so little attention thus far. One of the main aims of our work was to demonstrate that this area of mathematics admits many interesting questions and has the potential to become a vibrant area of research.

Our work is mainly centered around the monotonicity conjecture of \cite{baryshnikov}, which has turned out to be a more interesting question than initially thought, especially since it turned out to be false. This leaves us with many open questions. For instance, the constructions we provide rely on the existence of cycles in the superlevel sets of the functions. This leads one to wonder if there is a more conceptual reason explaining this failure of monotonicity in the presence of cycles and what are the precise conditions a function should satisfy for monotonicity to hold. It would be interesting to construct topological invariants measuring the extent to which monotonicity can fail.

We have reformulated the original results of Baryshnikov and Ghrist for functions $f:\RR\to[0,\infty)$ in a language that we feel is more natural than that of the original article, using the concepts of total, positive and negative variation. This has led us to a general decomposition theorem for such functions, as well as a characterization of functions $f:S^1\to[0,\infty)$. The question of what the natural context for a general treatment of $\ucat$ for continuous functions $f:\RR^n\to[0,\infty)$ might be, remains widely open. We speculate that the answer might lie in a new kind of (co)homology theory, designed to treat such problems in general.

\section{Acknowledgments}

The author would like to thank his thesis advisors Dušan Repovš and Primož Škraba, the latter of whom suggested the problem. Thanks also to Jaka Smrekar for encouragement during the early stages of this research. The author was supported by the Slovenian Research Agency grant P1-0292-0101.

\appendix

\section{Direct Descriptions of the Counterexamples in $\RR^2$}

\subsection{First Example}\label{direct1}

Piecewise linearity of the functions $u_1,u_2,F:\RR^2\to[0,\infty)$ appearing in Section \ref{plane1} follows from the properties of the $\infty$-distance. For concreteness, we explicitly describe the decomposition of the plane with respect to which the functions are piecewise linear. The main advantage of $u_1$ and $F$ being piecewise linear is that other facts about these functions can be verified completely computationally.

Before we begin, we need some notation. Two points $p_1,p_2\in\RR^2$ determine a segment
\[
p_1p_2=\{p\in\RR^2\mid\exists t\in[0,1]:p=(1-t)p_1+tp_2\}.
\]
We write $p_1p_2\ldots p_n$ for the union of segments $p_1p_2,p_2p_3,\ldots,p_{n-1}p_n$. If $p_1p_2\ldots p_np_1$ is a topological circle in $\RR^2$, it is the boundary of a uniquely determined compact set in $\RR^2$, which we denote by $\overline{p_1p_2\ldots p_n}$.

We can describe $u_1,u_2$ and $F$ as piecewise linear functions defined by their values on the vertices of a polygonal decomposition of $S=\supp f$ consisting of $44$ vertices, $95$ edges and $52$ faces, namely triangles, trapezoids and two non-convex quadrilaterals. Note that we do not count the ``face at infinity'' and we consider parallelograms to be a special case of trapezoids. Some care must be taken as not every choice of values at the vertices of a quadrilateral can be extended to a linear function. First, we list the vertices (indexed lexicographically):
\[
\begin{aligned}
x_1&=(-4, 0),\\
x_2&=(-4, 2),\\
x_3&=(-4, 4),\\
x_4&=(-3, 1),\\
x_5&=(-3, 2),\\
x_6&=(-3, 3),\\
x_7&=(-2, -2),\\
x_8&=(-2, 0),\\
x_9&=(-2, 1),\\
x_{10}&=(-2, 2),\\
x_{11}&=(-2, 3),
\end{aligned}\qquad\begin{aligned}
x_{12}&=(-2, 4),\\
x_{13}&=(-1, -1),\\
x_{14}&=(-1, 1),\\
x_{15}&=(-1, 2),\\
x_{16}&=(-1, 3),\\
x_{17}&=(0, -4),\\
x_{18}&=(0, -2),\\
x_{19}&=(0, 2),\\
x_{20}&=(0, 4),\\
x_{21}&=(1, -3),\\
x_{22}&=(1, -2),
\end{aligned}\qquad\begin{aligned}
x_{23}&=(1, -1),\\
x_{24}&=(1, 1),\\
x_{25}&=(1, 2),\\
x_{26}&=(1, \tfrac{11}5),\\
x_{27}&=(1, 3),\\
x_{28}&=(2, -4),\\
x_{29}&=(2, -3),\\
x_{30}&=(2, -2),\\
x_{31}&=(2, -1),\\
x_{32}&=(2, 0),\\
x_{33}&=(2, 1),
\end{aligned}\qquad\begin{aligned}
x_{34}&=(2, 2),\\
x_{35}&=(2, 4),\\
x_{36}&=(\tfrac{11}5, 3),\\
x_{37}&=(3, -3),\\
x_{38}&=(3, -2),\\
x_{39}&=(3, -1),\\
x_{40}&=(3, 1),\\
x_{41}&=(4, -4),\\
x_{42}&=(4, -2),\\
x_{43}&=(4, 0),\\
x_{44}&=(4, 2).
\end{aligned}
\]
We omit listing the edges $h_1,h_2,\ldots,h_{95}$ (ordered lexicographically by the indices of the vertices) as they are simply the edges of the $52$ polygons in the decomposition. Finally, we list the faces, using the notation defined above (ordered lexicographically by the corresponding sets of vertices):
\[
\mkern-24mu\begin{aligned}
f_1&=\overline{x_1x_4x_5x_2},\\
f_2&=\overline{x_1x_8x_9x_4},\\
f_3&=\overline{x_2x_6x_3},\\
f_4&=\overline{x_2x_5x_6},\\
f_5&=\overline{x_3x_6x_{12}},\\
f_6&=\overline{x_4x_{10}x_5},\\
f_7&=\overline{x_4x_9x_{10}},\\
f_8&=\overline{x_5x_{10}x_{11}x_6},\\
f_9&=\overline{x_6x_{11}x_{12}},\\
f_{10}&=\overline{x_7x_{13}x_{14}x_8},\\
f_{11}&=\overline{x_7x_{18}x_{23}x_{13}},\\
f_{12}&=\overline{x_8x_{14}x_9},\\
f_{13}&=\overline{x_9x_{14}x_{15}x_{10}},
\end{aligned}\qquad\begin{aligned}
f_{14}&=\overline{x_{10}x_{16}x_{11}},\\
f_{15}&=\overline{x_{10}x_{15}x_{16}},\\
f_{16}&=\overline{x_{11}x_{16}x_{20}x_{12}},\\
f_{17}&=\overline{x_{13}x_{23}x_{24}x_{14}},\\
f_{18}&=\overline{x_{14}x_{19}x_{15}},\\
f_{19}&=\overline{x_{14}x_{24}x_{19}},\\
f_{20}&=\overline{x_{15}x_{19}x_{20}x_{16}},\\
f_{21}&=\overline{x_{17}x_{21}x_{22}x_{18}},\\
f_{22}&=\overline{x_{17}x_{28}x_{29}x_{21}},\\
f_{23}&=\overline{x_{18}x_{22}x_{23}},\\
f_{24}&=\overline{x_{19}x_{27}x_{20}},\\
f_{25}&=\overline{x_{19}x_{24}x_{25}},\\
f_{26}&=\overline{x_{19}x_{25}x_{26}},
\end{aligned}\qquad\begin{aligned}
f_{27}&=\overline{x_{19}x_{26}x_{34}x_{27}},\\
f_{28}&=\overline{x_{20}x_{27}x_{35}},\\
f_{29}&=\overline{x_{21}x_{30}x_{22}},\\
f_{30}&=\overline{x_{21}x_{29}x_{30}},\\
f_{31}&=\overline{x_{22}x_{30}x_{31}x_{23}},\\
f_{32}&=\overline{x_{23}x_{32}x_{24}},\\
f_{33}&=\overline{x_{23}x_{31}x_{32}},\\
f_{34}&=\overline{x_{24}x_{33}x_{34}x_{25}},\\
f_{35}&=\overline{x_{24}x_{32}x_{33}},\\
f_{36}&=\overline{x_{25}x_{34}x_{26}},\\
f_{37}&=\overline{x_{27}x_{34}x_{35}},\\
f_{38}&=\overline{x_{28}x_{37}x_{29}},\\
f_{39}&=\overline{x_{28}x_{41}x_{37}},
\end{aligned}\qquad\begin{aligned}
f_{40}&=\overline{x_{29}x_{37}x_{38}x_{30}},\\
f_{41}&=\overline{x_{30}x_{39}x_{31}},\\
f_{42}&=\overline{x_{30}x_{38}x_{39}},\\
f_{43}&=\overline{x_{31}x_{39}x_{43}x_{32}},\\
f_{44}&=\overline{x_{32}x_{36}x_{33}},\\
f_{45}&=\overline{x_{32}x_{40}x_{34}x_{36}},\\
f_{46}&=\overline{x_{32}x_{43}x_{40}},\\
f_{47}&=\overline{x_{33}x_{36}x_{34}},\\
f_{48}&=\overline{x_{34}x_{40}x_{44}},\\
f_{49}&=\overline{x_{37}x_{42}x_{38}},\\
f_{50}&=\overline{x_{37}x_{41}x_{42}},\\
f_{51}&=\overline{x_{38}x_{42}x_{43}x_{39}},\\
f_{52}&=\overline{x_{40}x_{43}x_{44}}.
\end{aligned}
\]
We can now state the alternative descriptions of $u_1,u_2$ and $F$. The function $u_1$ can be defined on the vertices of the above decomposition:
\[
u_1(x_i)=\begin{cases}5;&i=40,\\1;&i=4,5,6,9,13,14,23,24,31,33,36,37,38,39,\\0;&\text{elsewhere.}\end{cases}
\]
Similarly, we have
\[
u_2(x_i)=\begin{cases}5;&i=27,\\1;&i=6,11,13,14,15,16,21,22,23,24,25,26,29,37,\\0;&\text{elsewhere.}\end{cases}
\]
Summing these, we obtain
\[
F(x_i)=\begin{cases}5;&i=27,40,\\2;&i=6,13,14,23,24,37,\\1;&i=4,5,9,11,15,16,21,22,25,26,29,31,33,36,38,39,\\0;&\text{elsewhere.}\end{cases}
\]
The decomposition (and the function $F$) can be pictured as follows:

\hbox{}

\begin{center}
\includegraphics[width=170pt]{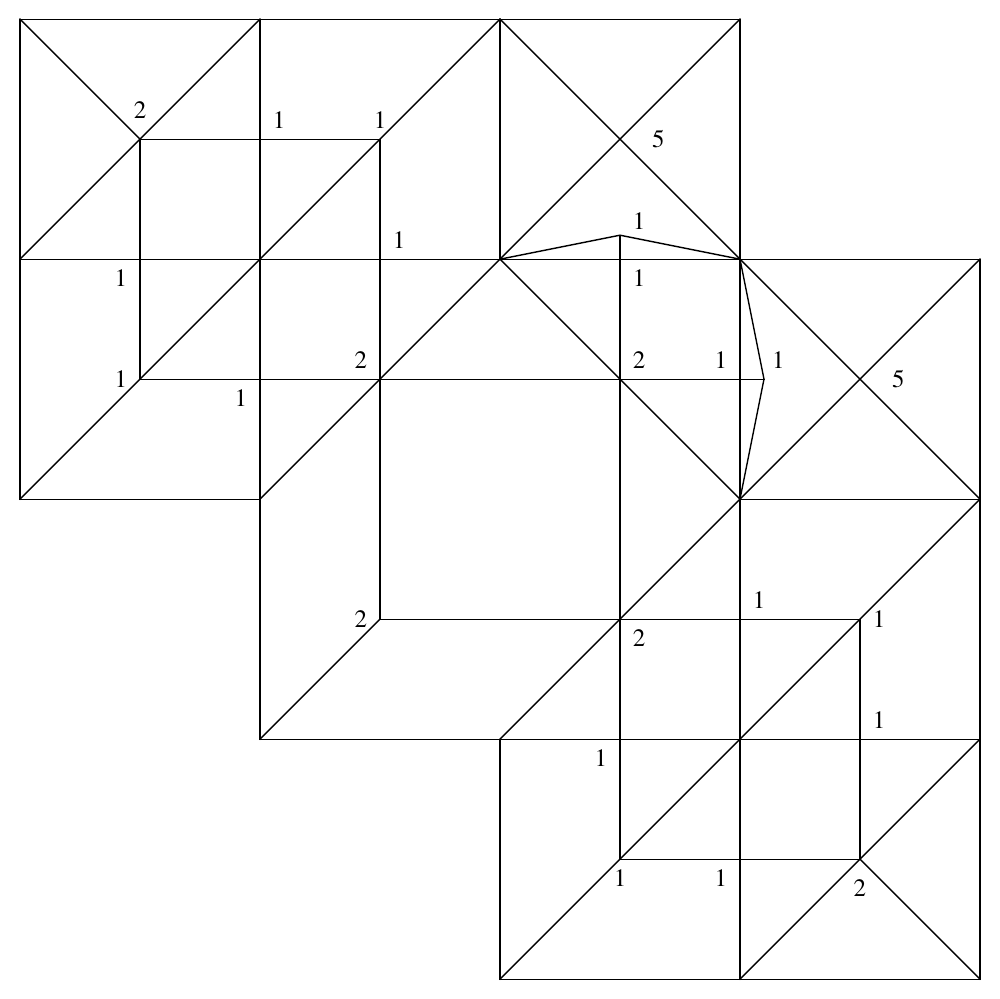}
\end{center}

\hbox{}

\subsection{Second Example}\label{direct2}

For the same reason as our first example, $f:\RR^2\to[0,\infty)$ is actually a piecewise linear function, so it can be described by its function values at the vertices of a polygonal decomposition of its support $S=\supp f$. This decomposition consists of $47$ vertices, $94$ edges and $46$ faces, namely triangles, trapezoids and two non-convex quadrilaterals. Again note that not every choice of values at the vertices of a trapezoid can be extended to a linear function, but in our case such issues do not arise as the two function values on each of two parallel sides agree. We first list the vertices (indexed lexicographically):
\[
\begin{aligned}
x_1&=(-7, -7),\\
x_2&=(-7, 7),\\
x_3&=(-6, -6),\\
x_4&=(-6, 6),\\
x_5&=(-5, -5),\\
x_6&=(-5, -1),\\
x_7&=(-5, 1),\\
x_8&=(-5, 5),\\
x_9&=(-4, 0),\\
x_{10}&=(-\tfrac{10}3, 0),\\
x_{11}&=(-3, -5),\\
x_{12}&=(-3, -1),
\end{aligned}\qquad\begin{aligned}
x_{13}&=(-3, 1),\\
x_{14}&=(-3, 5),\\
x_{15}&=(-2, -4),\\
x_{16}&=(-2, 0),\\
x_{17}&=(-2, 4),\\
x_{18}&=(-1, -5),\\
x_{19}&=(-1, -3),\\
x_{20}&=(-1, -1),\\
x_{21}&=(-1, 1),\\
x_{22}&=(-1, 3),\\
x_{23}&=(-1, 5),\\
x_{24}&=(0, -6),
\end{aligned}\qquad\begin{aligned}
x_{25}&=(0, -4),\\
x_{26}&=(0, 0),\\
x_{27}&=(0, 4),\\
x_{28}&=(0, 6),\\
x_{29}&=(1, -7),\\
x_{30}&=(1, -5),\\
x_{31}&=(1, -3),\\
x_{32}&=(1, -1),\\
x_{33}&=(1, 1),\\
x_{34}&=(1, 3),\\
x_{35}&=(1, 5),\\
x_{36}&=(1, 7),
\end{aligned}\qquad\begin{aligned}
x_{37}&=(2, -4),\\
x_{38}&=(2, 0),\\
x_{39}&=(2, 4),\\
x_{40}&=(3, -5),\\
x_{41}&=(3, -1),\\
x_{42}&=(3, 1),\\
x_{43}&=(3, 5),\\
x_{44}&=(\tfrac{10}3, 0),\\
x_{45}&=(4, 0),\\
x_{46}&=(5, -1),\\
x_{47}&=(5, 1).
\end{aligned}
\]
We again omit listing the edges $h_1,h_2,\ldots,h_{94}$ and proceed to the faces:
\[
\mkern-12mu\begin{aligned}
f_1&=\overline{x_1x_3x_4x_2},\\
f_2&=\overline{x_1x_{29}x_{24}x_3},\\
f_3&=\overline{x_2x_4x_{28}x_{36}},\\
f_4&=\overline{x_3x_5x_8x_4},\\
f_5&=\overline{x_3x_{24}x_{18}x_5},\\
f_6&=\overline{x_4x_8x_{23}x_{28}},\\
f_7&=\overline{x_6x_9x_7},\\
f_8&=\overline{x_6x_{12}x_9},\\
f_9&=\overline{x_7x_9x_{13}},\\
f_{10}&=\overline{x_9x_{12}x_{10}x_{13}},\\
f_{11}&=\overline{x_{10}x_{12}x_{16}},\\
f_{12}&=\overline{x_{10}x_{16}x_{13}},
\end{aligned}\qquad\begin{aligned}
f_{13}&=\overline{x_{11}x_{15}x_{16}x_{12}},\\
f_{14}&=\overline{x_{11}x_{18}x_{25}x_{15}},\\
f_{15}&=\overline{x_{13}x_{16}x_{17}x_{14}},\\
f_{16}&=\overline{x_{14}x_{17}x_{27}x_{23}},\\
f_{17}&=\overline{x_{15}x_{19}x_{20}x_{16}},\\
f_{18}&=\overline{x_{15}x_{25}x_{19}},\\
f_{19}&=\overline{x_{16}x_{21}x_{22}x_{17}},\\
f_{20}&=\overline{x_{16}x_{20}x_{26}},\\
f_{21}&=\overline{x_{16}x_{26}x_{21}},\\
f_{22}&=\overline{x_{17}x_{22}x_{27}},\\
f_{23}&=\overline{x_{18}x_{24}x_{25}},\\
f_{24}&=\overline{x_{19}x_{25}x_{26}x_{20}},
\end{aligned}\qquad\begin{aligned}
f_{25}&=\overline{x_{21}x_{26}x_{27}x_{22}},\\
f_{26}&=\overline{x_{23}x_{27}x_{28}},\\
f_{27}&=\overline{x_{24}x_{29}x_{30}x_{25}},\\
f_{28}&=\overline{x_{25}x_{31}x_{32}x_{26}},\\
f_{29}&=\overline{x_{25}x_{30}x_{40}x_{37}},\\
f_{30}&=\overline{x_{25}x_{37}x_{31}},\\
f_{31}&=\overline{x_{26}x_{33}x_{34}x_{27}},\\
f_{32}&=\overline{x_{26}x_{32}x_{38}},\\
f_{33}&=\overline{x_{26}x_{38}x_{33}},\\
f_{34}&=\overline{x_{27}x_{35}x_{36}x_{28}},\\
f_{35}&=\overline{x_{27}x_{34}x_{39}},
\end{aligned}\qquad\begin{aligned}
f_{36}&=\overline{x_{27}x_{39}x_{43}x_{35}},\\
f_{37}&=\overline{x_{31}x_{37}x_{38}x_{32}},\\
f_{38}&=\overline{x_{33}x_{38}x_{39}x_{34}},\\
f_{39}&=\overline{x_{37}x_{40}x_{41}x_{38}},\\
f_{40}&=\overline{x_{38}x_{41}x_{44}},\\
f_{41}&=\overline{x_{38}x_{42}x_{43}x_{39}},\\
f_{42}&=\overline{x_{38}x_{44}x_{42}},\\
f_{43}&=\overline{x_{41}x_{45}x_{42}x_{44}},\\
f_{44}&=\overline{x_{41}x_{46}x_{45}},\\
f_{45}&=\overline{x_{42}x_{45}x_{47}},\\
f_{46}&=\overline{x_{45}x_{46}x_{47}}.
\end{aligned}
\]
An alternative definition of the piecewise linear function $f:\RR^2\to[0,\infty)$ is then given by specifying it on the vertices as follows:
\[
f(x_i)=\begin{cases}3;&i=9,45,\\1;&i=3,4,10,15,16,17,24,25,26,27,28,37,38,39,44,\\0;&\text{otherwise.}\end{cases}
\]

The decomposition (and the function) can be pictured as follows:

\hbox{}

\begin{center}
\includegraphics[width=170pt]{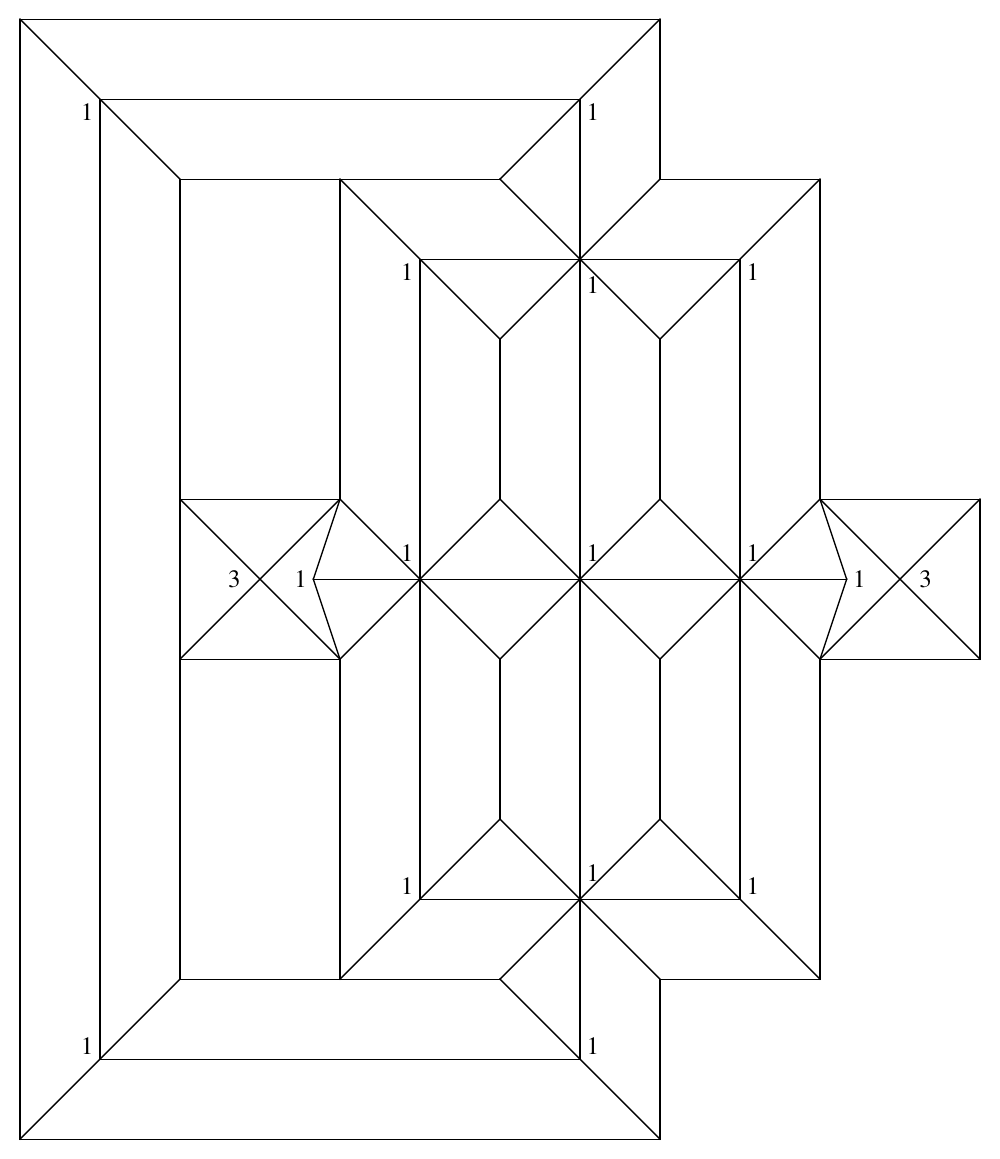}
\end{center}

\hbox{}

\begin{proposition}
The function $f:\RR^2\to[0,\infty)$ has a unimodal $\infty$-decomposition of length $2$.
\end{proposition}

\begin{proof}
Define two piecewise linear functions $v_1,v_2:\RR^2\to[0,\infty)$ on a polygonal decomposition of $S$ by specifying their values at the vertices. Start with the decomposition of $S$ defined above. Further subdivide it by adding the edges $x_{15}x_{18}, x_{17}x_{23}, x_{30}x_{37}, x_{35}x_{39}$. The function $v_1$ assumes the value $1$ at the vertices $x_3,x_4,x_{10},x_{15},x_{16},x_{17},x_{24},x_{25},x_{26},x_{27},x_{28}$; the value $3$ at the vertex $x_9$; and the value $0$ at all the other vertices. The function $v_2$ assumes the value $1$ at the vertices $x_3,x_4,x_{24},x_{25},x_{26},x_{27},x_{28},x_{37},x_{38},x_{39},x_{44}$; the value $3$ at the vertex $x_{45}$; and the value $0$ at all the other vertices. It is clear that $\max\{v_1,v_2\}=f$, however, these functions are not unimodal, as $x_3x_{24}x_{28}x_4x_3$ yields a non-trivial cycle in some of the superlevel sets. However, we can modify $v_1$ and $v_2$ to obtain unimodal functions $u_1,u_2:\RR^2\to[0,\infty)$. To retain the property $\max\{u_1,u_2\}=f$, we modify them on sets with disjoint interiors, namely, $u_1$ is modified on the sets $R_1=[-1,1]\times[2,3]$, $R_2=[-1,1]\times[-2,-1]$ and $R_3=[-7,-5]\times[-2,-1]$, whereas $u_2$ is modified on the sets $R_4=[-1,1]\times[1,2]$, $R_5=[-1,1]\times[-3,-2]$ and $R_6=[-7,-5]\times[1,2]$. In fact, they are modified in the same way on each of these. Namely, if $R$ is of one of these rectangles, subdivide it into four triangles using the center point of $R$. Then define a piecewise linear function on $R$ that takes the value $1$ at the center point of $R$ and value $0$ at the vertices of the rectangle, and extend it by $0$ to the whole plane to obtain a function $\varphi_R:\RR^2\to[0,\infty)$. Now, $u_1$ and $u_2$ are defined by putting $u_1=v_1-\varphi_{R_1}-\varphi_{R_2}-\varphi_{R_3}$ and $u_2=v_2-\varphi_{R_4}-\varphi_{R_5}-\varphi_{R_6}$. (These functions are piecewise linear w.r.t. a further subdivision of $S$ that takes into account the six rectangles.) A straightforward verification shows that $u_1$ and $u_2$ are unimodal, but it is ultimately unenlightening, so instead of this, we provide the graphs of the two functions.
\end{proof}

\begin{center}
\includegraphics[width=170pt]{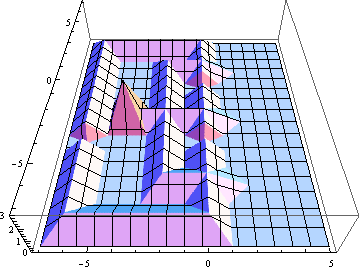}\qquad\includegraphics[width=170pt]{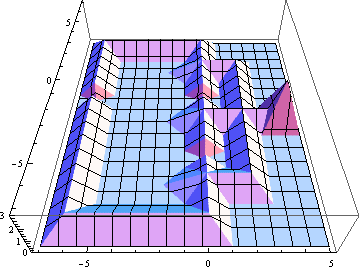}
\end{center}

\begin{proposition}
The function $f:\RR^2\to[0,\infty)$ has a unimodal decomposition of length $3$.
\end{proposition}

\begin{proof}
An explicit unimodal decomposition of $f:\RR^2\to[0,\infty)$ can be described as follows: take the polygonal decomposition of $S$ as defined in the beginning. Further subdivide it by adding the edges $x_1x_4, x_4x_5, x_{15}x_{18}, x_{20}x_{25}, x_{21}x_{27}, x_{25}x_{32}, x_{27}x_{33}, x_{35}x_{39}$. Now, define piecewise linear functions $u_1,u_2,u_3:\RR^2\to[0,\infty)$ by specifying their values at the vertices of this decomposition. For the sake of brevity, we only specify the nonzero values. The function $u_1$ is defined by taking the value $3$ at the vertex $x_9$ and the value $1$ at the vertices $x_4,x_{10},x_{15},x_{16},x_{17},x_{27},x_{28}$. The function $u_2$ is defined by taking the value $3$ at the vertex $x_{45}$ and the value $1$ at the vertices $x_3,x_{24},x_{25},x_{37},x_{38},x_{39},x_{44}$. Finally, $u_3$ takes the value $1$ at the vertex $x_{26}$. It is straightforward to verify that these are all unimodal and that $f=u_1+u_2+u_3$. 
\end{proof}

\begin{center}
\includegraphics[width=170pt]{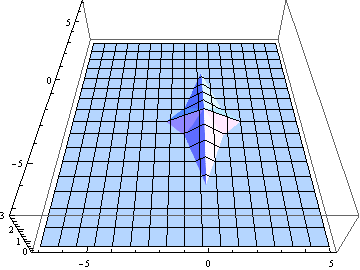}\qquad\includegraphics[width=170pt]{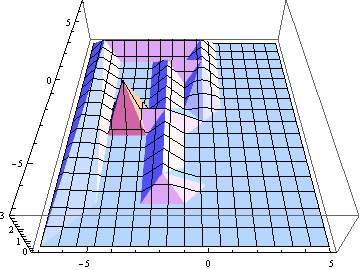}\qquad\includegraphics[width=170pt]{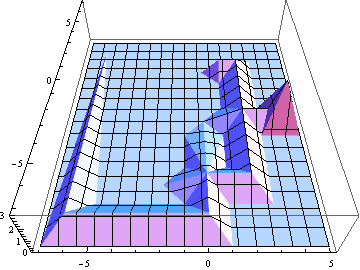}
\end{center}

\section{Algorithm for the Circle}\label{algorithms}

Theorem \ref{sweeping} provides a generalization of the sweeping algorithm computing $\ucat(f)$ for $f:\RR\to[0,\infty)$ with finitely many critical points, given in \cite{baryshnikov}. Here we describe an algorithm to compute $\ucat(f)$ for $f:S^1\to[0,\infty)$ in the case of finitely many critical points. Although this is possible, as the proofs of the theorems leading to the algorithm are constructive in nature, we do not compute the explicit decomposition.

\begin{algorithm}
\caption{Computing $\ucat$ for the Circle}
\begin{algorithmic}
\Require $f:S^1\to[0,\infty)$ without zeros and with critical points $x_1,\ldots,x_k\in S^1$.
\For{$i=1,\ldots,k$}
\State $f_i(t):=\begin{cases}
(t+1)f(x_i);& t\in[-1,0],\\
f(x_i e^{2\pi i t});&t\in[0,1],\\
(2-t)f(x_i);&t\in[1,2],\\
0;&t\notin[-1,2].
\end{cases}$
\EndFor
\For{$i=1,\ldots,k$}
\State $j=0$
\State $t_0=-\infty$
\While{$t_j<1$}
\State $j:=j+1$
\State $t_j=\min\{(t_{j-1},t)\mid V^-(f_i;(t_{j-1},t))>f_i(t_{j-1})\}$
\EndWhile
\State $\alpha_i:=j-1$
\EndFor
\Return $\min_{1\leq i\leq k}\alpha_i$.
\end{algorithmic}
\end{algorithm}

\end{document}